\documentclass[10pt,a4paper]{article}

\usepackage[utf8]{inputenc}
\usepackage[T1]{fontenc}
\usepackage[english]{babel}
\usepackage[left=2.00cm,right=2.00cm,top=2.00cm,bottom=2.00cm]{geometry}
\usepackage{multicol}
\usepackage{amsmath}
\usepackage{amsfonts}
\usepackage{amssymb}
\usepackage{amsthm}
\usepackage{mathtools}
\usepackage{graphicx}
\usepackage[percent]{overpic}
\usepackage{xcolor}
\usepackage{caption}
\usepackage{subcaption}
\usepackage{booktabs}
\usepackage{array}
\usepackage{tabularx}
\usepackage{float}
\usepackage{xurl}

\definecolor{azul}{rgb}{0.0,0.53,0.74}

\usepackage[
  colorlinks=true,
  linkcolor=azul,
  citecolor=azul!80!black,
  urlcolor=azul
]{hyperref}

\usepackage[numbers,sort&compress]{natbib}

\raggedcolumns

\theoremstyle{plain}
\newtheorem{theorem}{Theorem}[section]

\newtheorem{proposition}[theorem]{Proposition}

\theoremstyle{definition}

\newcommand{\tapsrepo}{\href{https://github.com/liminae/TAPS/tree/TAPS}{\texttt{github.com/liminae/TAPS}}}
\newcommand{\repro}[2]{\space\textit{Reproduction:} \href{https://github.com/liminae/TAPS/blob/TAPS/notebooks/#1}{\texttt{#1}, Cell~#2}.}
\newcommand{\repros}[2]{\space\textit{Reproduction:} \href{https://github.com/liminae/TAPS/blob/TAPS/notebooks/#1}{\texttt{#1}, Cells~#2}.}

\newcommand{\norm}[1]{\left\lVert #1 \right\rVert}
\newcommand{\blueline}{
  \begin{center}
  \textcolor{azul}{\rule{150mm}{0.5mm}}
  \end{center}
}

\begin{document}

\vspace{5mm}

\begin{center}
{\Large\textbf{TAPS: Target-Aware Permanent Sampling for Graph Diffusion}}\\[5mm]
{\normalsize Michelle Lin$^{1}$ \quad Laura P.\ Schaposnik$^{2,3,4}$}\\[3mm]
$^{1}${Thomas Jefferson High School for Science and Technology}, VA, USA\\[1mm]
$^{2}$MSCS, LQuTE, AIRRSHIP Lab, 
University of Illinois Chicago, USA\\[1pt]
$^{3}$NSF-Simons National Institute for Theory and Mathematics in
Biology, Chicago, IL, USA\\[1pt]
 
$^{4}$Mathematical Institute and Magdalen College,
University of Oxford, UK\\[4pt]
\end{center}

\blueline

\begin{abstract}
Permanent air-quality networks are expensive to install and maintain, yet many decisions depend on one future regional exposure rather than the complete pollution field. We introduce \textit{Target-Aware Permanent Sampling (TAPS)}, a graph-diffusion framework for selecting permanent locations observed repeatedly over time,  that can reduce unnecessary sensor installations, maintenance, and cost while preserving the information needed for future air-quality decisions, subject to validation with the estimator that will use the network. We formulate the permanent space-time sampling problem and show that, in the noiseless model, recovery of one prescribed target can require less information than full-state identification, for which we give an explicit lower bound on the permanent-location count. We also derive a greedy rule whose marginal gain factors into raw target response and a finite-update correction.  We evaluate TAPS on regulatory air-quality networks in California, Canada, and England. In each case the prescribed regional target becomes numerically recoverable at permanent-location budgets well below those required to identify the retained state, and TAPS attains lower regularized target risk than target-weight, geometric, and design-based placements. A blind prospective study further shows that the criterion can be applied before candidate sites have any ground-monitor history, using exogenous environmental covariates alone. 

\vspace{2mm}
\noindent\textbf{Keywords:}
\textit{graph sampling, graph diffusion, permanent sensor placement,
target-aware sampling, air-quality monitoring, optimal experimental design.}
\end{abstract}

\blueline
\vspace{4mm}

\begin{multicols}{2}

\section{Introduction}
\label{sec:intro}

Wildfire smoke creates a public-health burden whose spatial and temporal variation is difficult to resolve with a finite permanent monitoring network. Across the contiguous United States, wildfire PM$_{2.5}$ exposure was estimated to cause 164,000 deaths from 2006 to 2020 \cite{law2025wildfire}; in California, exposure from 2008 to 2018 was estimated to cause up to 55,710 premature deaths and \$456 billion in economic damages \cite{connolly2024mortality}. Forest-fire models have long been studied in statistical mechanics as idealized systems for self-organized criticality \cite{bak1990forest,drossel1992forest,clar1996forest}. Some observed fire frequency--area distributions show approximate power-law behavior \cite{malamud1998forest}, although later analyses question simple scaling in the Drossel--Schwabl model \cite{pruessner2002broken}. These results concern fire sizes and dynamics, not the distribution of ground-level PM$_{2.5}$ or atmospheric smoke transport. Our question concerns monitoring rather than fire spread: how many permanent locations are needed for one future regional exposure rather than the complete pollution field? 

Most PM$_{2.5}$ models estimate concentrations after an observation network is available \cite{power2024comparison}. Sensor-placement methods instead choose observations for objectives such as field reconstruction, spatial coverage, or demographic priorities \cite{zhou2022pmplacement,kelp2023placement,YS26}. On graphs, bandlimited, QR, and dynamical sampling seek recovery of a retained signal class \cite{anis2016sampling,manohar2018sparse,aldroubi2013dynamical,huang2025spacetime}; graphical, $c$-optimal, goal-oriented, and task-oriented designs target a prescribed functional or downstream task \cite{babecki2023graphical,althani2024sparse,pukelsheim2006optimal,attia2018goal,madhavan2026control}. Functional observability likewise studies recovery of selected state quantities without requiring full-state observability, including the associated sensor-placement problem \cite{montanari2022functional,zhang2025functional}. Permanent monitoring adds a distinct cost structure: one installed location can be observed repeatedly, whereas every additional location requires new infrastructure, see  \cite{epaSensorGuide2022}.

\begin{center}
\includegraphics[width=1\linewidth]{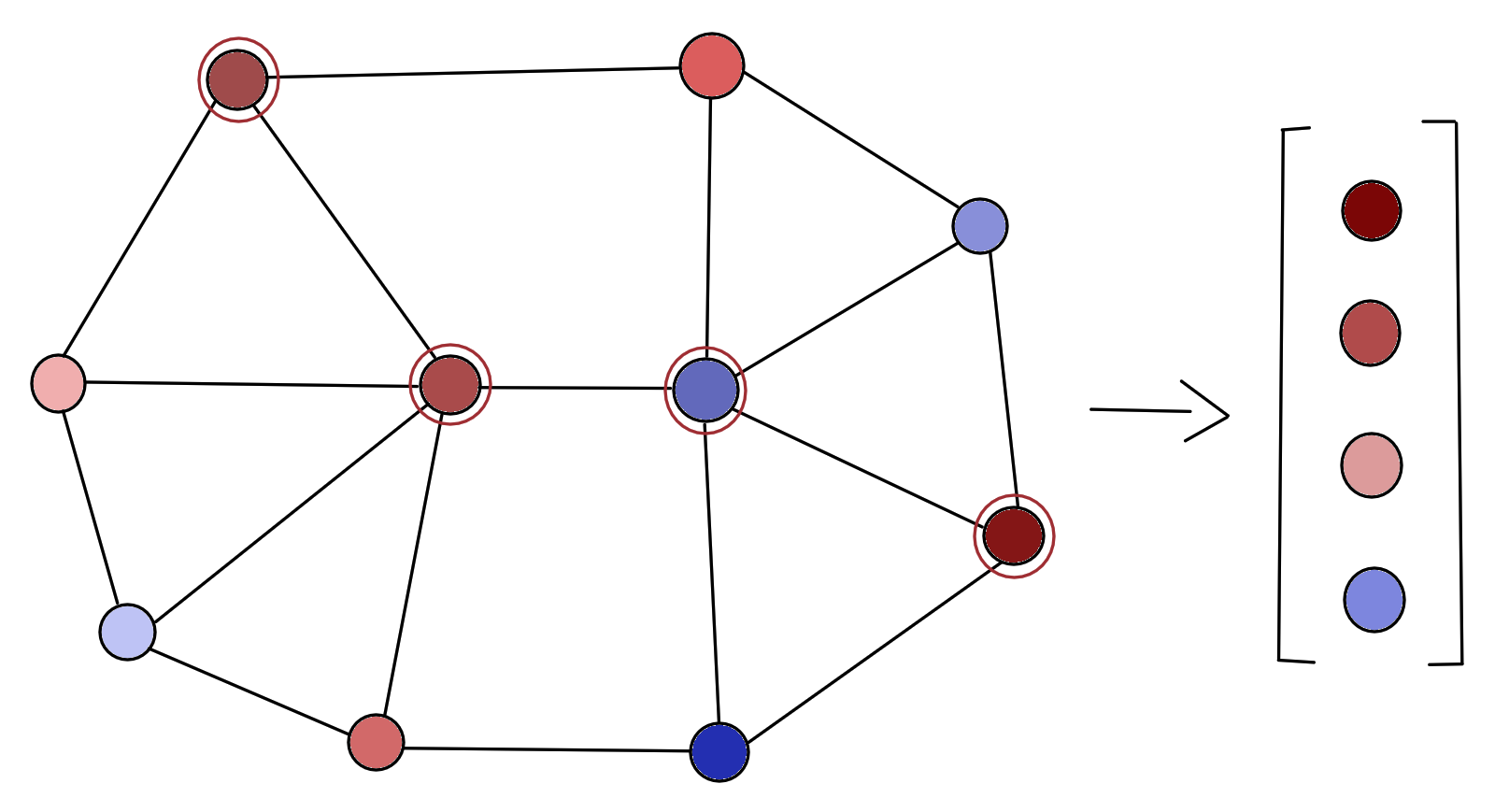}
\captionof{figure}{Schematic of TAPS. Red rings denote the installed set $S$. At one observation time, the installed locations produce the measurement vector shown at right; repeated observation times stack these measurements into the blocks $R_v$. TAPS selects permanent locations for one prescribed future target rather than for reconstruction of the complete retained state.}
\label{fig:graph-sampling}
\end{center}
\pagebreak
In this work we introduce Target-Aware Permanent Sampling (TAPS) for this permanent space-time problem, summarized in  Figure~\ref{fig:graph-sampling} above. The central idea of TAPS is that a location is valuable not only when it responds strongly to the target, but when it contributes target information not already supplied by the installed network. Graph diffusion represents each selected vertex by a structured block of repeated measurements, allowing one future functional to be compared directly with the stronger objective of full-state identification.

We note at the outset that the selection criterion itself is the grouped, permanent-location analogue of greedy $c$-optimal design \cite{pukelsheim2006optimal}: under the same criterion and tie-breaking the two make identical selections. The contribution is the permanent space-time formulation in which each decision commits to an entire repeated-observation block at the cost of one physical installation, together with the target-versus-state comparison, the marginal-gain factorization, and prospective placement without candidate-site histories that this formulation makes available. 

Our main results can be summarized as follows:
\begin{itemize}
    \item First, the target can be recovered before the retained state is identifiable: ten California locations recover the Bay Area target with rank $28<40$, while in England the South East target is numerically exact at ten locations and the TAPS sequence first reaches full-state rank at 22 (Figure~\ref{fig:risk-conditioning} and Section~\ref{sec:transfer}).
    \item Second, under the tested rescaled observation schedules, increasing readings per location from one to five lowers the California greedy exact-target budget from 33 locations to seven, and including the finite-update correction changes six of the ten selected locations (Figures~\ref{fig:observation-count} and~\ref{fig:structural-results}).
    \item  Third, TAPS can operate before candidate sites have ground-monitor PM$_{2.5}$ histories. Across 100 blind prospective California histories, it has the lowest tested ten-location $r_\beta$ in 92; target-weight and raw-response placement require median budgets of 15 and 17 locations to match TAPS-10, so TAPS uses 33.3\% and 41.2\% fewer permanent monitors (Figure~\ref{fig:blind-prospective}). 
    \item Finally, held-out experiments show that model-risk and forecast-error rankings are not identical: ten-location TAPS is competitive with all tested alternatives through 30 locations, but the post hoc best-competitor comparison is unresolved, and at 30 scheduled readings graph diffusion favors $30\times1$ whereas random forest favors $10\times3$ (Figure~\ref{fig:heldout-budget-curves} and Table~\ref{tab:equal-reading-budget}). 
\end{itemize} Together, the theory and experiments show that TAPS can reduce the permanent infrastructure needed for a prescribed future quantity, while the final network must still be tested with the estimator that will use it.

\section{Mathematical Framework}
\label{sec:background}
\label{sec:model}

We formulate TAPS as a permanent space-time sampling problem. The graph, target, prediction horizon, observation schedule, and candidate locations are fixed before future measurements are observed. Selecting one vertex adds its complete block of repeated measurements, but incurs the cost of one physical installation.

\subsection{Graph Diffusion and Measurements}

Let $G=(V,E,\omega)$ be a finite, undirected graph with $n=|V|$ vertices and edge weights $\omega_{uv}>0$ for $\{u,v\}\in E$, and assume $G$ is connected. Let $W$ be the weighted adjacency matrix, $W_{uv}=\omega_{uv}$ if $\{u,v\}\in E$ and $W_{uv}=0$ otherwise, and let $D$ be the diagonal matrix of weighted degrees, $D_{vv}=\sum_u \omega_{uv}$. We work throughout with the unnormalized weighted Laplacian $L=D-W$. Since $L$ is real and symmetric, it admits an orthogonal eigendecomposition
\begin{eqnarray}
      L = U\Lambda U^{T},
  \qquad
  \Lambda := \operatorname{diag}(\lambda_0,\ldots,\lambda_{n-1}),\label{eig}
\end{eqnarray}
where $U$ is orthogonal and its columns $\phi_0,\ldots,\phi_{n-1}$ form an orthonormal graph-Fourier basis. For a vertex $v$, we write $\phi_\ell(v)$ for the entry of $\phi_\ell$ at $v$. Since the weights are nonnegative, $L$ is positive semidefinite, and since $G$ is connected the eigenvalues, ordered increasingly, satisfy
\[
  0 = \lambda_0 < \lambda_1 \leq \cdots \leq \lambda_{n-1}.
\] Small eigenvalues correspond to graph-smooth spatial modes \cite{chung1997spectral}.

For a retained bandwidth $K$, write
\begin{eqnarray}
    U_K&:=&
\begin{bmatrix}
\phi_0&\cdots&\phi_{K-1}
\end{bmatrix},\nonumber \\
\Lambda_K&:=&\operatorname{diag}(\lambda_0,\ldots,\lambda_{K-1}). \label{uk}
\end{eqnarray}
A retained initial state $x_0=U_K\alpha$ evolves under graph diffusion as
\begin{eqnarray}
    x(t)=e^{-tL}x_0=U_Ke^{-t\Lambda_K}\alpha,
\qquad
\alpha\in\mathbb{R}^K.\label{diff}
\end{eqnarray}
Higher graph frequencies decay more rapidly, while the constant mode is preserved.

Fix a prediction horizon $\tau>0$, the future time at which the target of Section~\ref{sec:target} is evaluated, and observation times $\mathcal{T}=\{t_1,\ldots,t_r\}$ with $0\leq t_1<\cdots<t_r<\tau$, so that $r=|\mathcal{T}|$ is the number of readings per installed location. The repeated measurements contributed by a permanent location $v$ are represented in the retained coordinates by the block
\begin{eqnarray}
R_v(j,\ell)=e^{-t_j\lambda_\ell}\phi_\ell(v),
\qquad
\substack{j=1,\ldots,r,\\
\ell=0,\ldots,K-1.}
\label{rv}
\end{eqnarray}
For a permanent set $S=\{v_1,\ldots,v_m\}$, define
\begin{eqnarray}
B_S:=
\begin{bmatrix}
R_{v_1}\\
\vdots\\
R_{v_m}
\end{bmatrix}
\in\mathbb{R}^{rm\times K},
\qquad
y=B_S\alpha+\varepsilon,
    \label{bs}
\end{eqnarray}
where $\varepsilon$ denotes measurement noise, for $B_S$ as in \eqref{bs}. In particular,  $m$ installations provide $rm$ scalar measurements, but the physical-location cost remains $m$. Here $\beta>0$, introduced in Section~\ref{sec:stable}, is a regularization parameter rather than an assumed noise variance; exact recovery below is explicitly noiseless. For zero-mean noise with $\operatorname{Cov}(\varepsilon)=\sigma^2 I$, $J_{\sigma^2}(S)$ of \eqref{jb} is the minimum worst-case mean squared error among linear estimators over $\|\alpha\|_2\leq1$. With correlated errors, $\operatorname{Cov}(\varepsilon)=\Sigma$, the variance term becomes $w^T\Sigma w$, so the effect on temporal reuse depends on the covariance and estimator. 

This construction places TAPS between two neighboring sampling objectives. Bandlimited graph sampling seeks enough measurements to reconstruct every $\alpha\in\mathbb{R}^K$, and therefore requires full column rank \cite{anis2016sampling}. Graphical designs reproduce one prescribed functional on a spectral space \cite{babecki2023graphical,althani2024sparse}. TAPS retains this target-versus-state distinction while requiring every decision to select an entire repeated-observation block $R_v$.

\subsection{Target Recovery and State Identification}
\label{sec:target}

Let
$
c\in\mathbb{R}^n,$
$
c\geq0,
$ and $
\mathbf{1}^Tc=1,
$
where $\mathbf{1}$ is the all-ones vector, so that $c$ is a probability weighting of the vertices. With $\tau$ the prediction horizon fixed above, the target functional is
\begin{eqnarray}
    F_{\tau,c}(x_0)
:=c^Tx(\tau)
=g_{\tau,c}^T\alpha,
\qquad
g_{\tau,c}:=e^{-\tau\Lambda_K}U_K^Tc. \label{gt}
\end{eqnarray}
We first distinguish the information needed for this one direction from the information needed for every retained direction.

\begin{proposition}[Exact target recovery]
\label{prop:exact}
In the noiseless setting, a linear estimator $\widehat F=w^Ty$ recovers $F_{\tau,c}(x_0)$ for every $\alpha\in\mathbb{R}^K$ if and only if
\[
g_{\tau,c}\in\operatorname{range}(B_S^T),
\]
for $g_{\tau,c}$ as in \eqref{gt}. 
Equivalently, exact recovery holds if and only if there exists $w\in\mathbb{R}^{r|S|}$ such that
\[
B_S^Tw=g_{\tau,c}.
\]
\end{proposition}

\begin{proof}
In the noiseless setting,
\[
\widehat F-F_{\tau,c}
=\bigl(B_S^Tw-g_{\tau,c}\bigr)^T\alpha.
\]
This vanishes for every $\alpha$ precisely when one has  $B_S^Tw=g_{\tau,c}$.
\end{proof}

We measure numerical exactness by the normalized target-span error
\begin{eqnarray}
e_{\mathrm{tar}}(S)
:=
\frac{\min_w\|B_S^Tw-g_{\tau,c}\|_2}
{\|g_{\tau,c}\|_2}.\label{etar}
\end{eqnarray}
In the experiments, we call a target numerically exact at tolerance $10^{-6}$ when $e_{\mathrm{tar}}(S)\leq10^{-6}$. Numerical rank counts singular values above $10^{-9}$, while target-span error uses the default cutoff in \texttt{np.linalg.lstsq} with \texttt{rcond=None}; the two diagnostics therefore use different numerical thresholds. The greedy exact-target budget is the first TAPS prefix satisfying the target-span criterion. It is a constructive budget for the TAPS sequence, not a proof of the smallest exact subset.

Full-state identification is stronger. For each retained eigenvalue $\lambda$, let $\mu_K(\lambda)$ be its multiplicity among $\lambda_0,\ldots,\lambda_{K-1}$, and set
\begin{eqnarray}
\mu_{\max}:=\max_\lambda\mu_K(\lambda).\label{mumax}
\end{eqnarray}

\begin{proposition}[Full-state lower bounds]
\label{prop:rank-lower}
If the complete retained state is identifiable from $r$ observations at each permanent vertex, then
\[
|S|\geq
\max\left\{
\left\lceil\frac{K}{r}\right\rceil,
\mu_{\max}
\right\}.
\]
\end{proposition}

\begin{proof}
Since $B_S$ has $r|S|$ rows,
\[
\operatorname{rank}(B_S)\leq r|S|,
\]
so full-state identification requires $|S|\geq\lceil K/r\rceil$. For a retained eigenvalue $\lambda$, the restriction of one block to its eigenspace has the outer-product form
\[
R_v^{(\lambda)}=
\begin{bmatrix}
e^{-t_1\lambda}\\
\vdots\\
e^{-t_r\lambda}
\end{bmatrix}
\begin{bmatrix}
\phi_\ell(v)
\end{bmatrix}_{\lambda_\ell=\lambda}^{T},
\]
where $\bigl[\phi_\ell(v)\bigr]_{\lambda_\ell=\lambda}$ denotes the vector of entries $\phi_\ell(v)$ over those indices $\ell$ with $\lambda_\ell=\lambda$. Therefore $R_v^{(\lambda)}$ (as in \eqref{rv}) has rank at most one. An eigenspace of multiplicity $\mu_K(\lambda)$ requires at least that many distinct vertices. Taking the largest multiplicity gives the second bound.
\end{proof}

Proposition~\ref{prop:rank-lower} identifies the remaining need for spatial diversity. Repeated observations can enlarge the measurement row space, but one permanent vertex still contributes at most one independent spatial direction within a repeated eigenspace.

For every retained model used below, $\mu_{\max}=1$, so the binding bound in Proposition~\ref{prop:rank-lower} is $\lceil K/r\rceil$.

\subsection{Stable Target Estimation}
\label{sec:stable}

Exact representation alone does not control sensitivity to noise. We therefore define, for $\beta>0$,
\begin{eqnarray}
    J_\beta(S)
:=
\min_w
\left\{
\|B_S^Tw-g_{\tau,c}\|_2^2
+
\beta\|w\|_2^2
\right\}.\label{jb}
\end{eqnarray}
The first term measures target mismatch, while the second penalizes estimator amplification.

\begin{proposition}[Regularized target risk]
\label{prop:risk}
The minimizing estimator is
\begin{eqnarray}
    w_\beta:=
\left(B_SB_S^T+\beta I\right)^{-1}B_Sg_{\tau,c},\label{wbeta}
\end{eqnarray}
and
\[
J_\beta(S)
=
\beta g_{\tau,c}^T
\left(B_S^TB_S+\beta I\right)^{-1}
 g_{\tau,c}.
\]
\end{proposition}

\begin{proof}
The normal equations give
\[
\left(B_SB_S^T+\beta I\right)w=B_Sg_{\tau,c}.
\]
Substitution of their solution, followed by the push-through identity, yields the stated expression for $J_\beta(S)$.
\end{proof}

We normalize this objective by
\begin{eqnarray}
 r_\beta(S;g_{\tau,c})^2
:=
\frac{J_\beta(S)}{\|g_{\tau,c}\|_2^2}
=
b_\beta(S)^2+a_\beta(S)^2,\label{rbeta}   
\end{eqnarray}
where
\begin{eqnarray}
b_\beta(S)
&:=&
\frac{\|B_S^Tw_\beta-g_{\tau,c}\|_2}
{\|g_{\tau,c}\|_2},\nonumber\\
a_\beta(S)
&:=&
\frac{\sqrt{\beta}\|w_\beta\|_2}
{\|g_{\tau,c}\|_2}.\label{ab}
\end{eqnarray}

Thus $b_\beta$ measures target mismatch and $a_\beta$ measures amplification.

Write $\sigma_{\min}(B_S)$ for the smallest singular value of $B_S$. To compare target estimation with state identification, define the worst-direction full-state risk
\begin{eqnarray}
    r_{\mathrm{full},\beta}(S)^2
:=
\frac{\beta}{\sigma_{\min}(B_S)^2+\beta},\label{rfull}
\end{eqnarray}
with $\sigma_{\min}(B_S)=0$ when $B_S$ is rank deficient.

\begin{proposition}[Target and full-state risk]
\label{prop:comparison}
For every nonzero target $g_{\tau,c}$,
\[
r_\beta(S;g_{\tau,c})
\leq
r_{\mathrm{full},\beta}(S).
\]
\end{proposition}

\begin{proof}
The squared target risk is the Rayleigh quotient of
\[
\beta\left(B_S^TB_S+\beta I\right)^{-1}
\]
at $g_{\tau,c}$. Its maximum over all nonzero directions is $\beta/(\sigma_{\min}(B_S)^2+\beta)$.
\end{proof}

Proposition~\ref{prop:comparison} shows that controlling every retained direction is sufficient, but not necessary, for controlling the target. The target may also occupy substantially fewer spectral directions: for $0<\delta<1$, we define $d_{\mathrm{eff}}(\delta)$ as
\begin{eqnarray}
\min\left\{
|\mathcal{J}|:
\sum_{\ell\in\mathcal{J}}|g_{\tau,c}(\ell)|^2
\geq
(1-\delta^2)\|g_{\tau,c}\|_2^2
\right\}. \label{deff}
\end{eqnarray}
This effective target dimension measures spectral concentration; it does not by itself guarantee that the available measurement blocks represent those modes efficiently.

We next formalize the value of reusing installed hardware. When the observation schedule must be shown explicitly, write $e_{\mathrm{tar},\mathcal{T}}$, $r_{\beta,\mathcal{T}}$, and $r_{\mathrm{full},\beta,\mathcal{T}}$.

\begin{proposition}[Additional repeated observations]
\label{prop:time-monotonicity}
Fix $S$, $g_{\tau,c}$, the retained model, and $\beta>0$. If $\mathcal{T}\subseteq\mathcal{T}'$, then
\[
\begin{aligned}
e_{\mathrm{tar},\mathcal{T}'}(S)
&\leq e_{\mathrm{tar},\mathcal{T}}(S),\\
r_{\beta,\mathcal{T}'}(S;g_{\tau,c})
&\leq r_{\beta,\mathcal{T}}(S;g_{\tau,c}),\\
r_{\mathrm{full},\beta,\mathcal{T}'}(S)
&\leq r_{\mathrm{full},\beta,\mathcal{T}}(S).
\end{aligned}
\]
\end{proposition}

\begin{proof}
Writing the enlarged measurement matrix as
\[
B_S'=
\begin{bmatrix}
B_S\\E_S
\end{bmatrix},
\]
where $E_S$ collects the rows contributed by the additional observation times $\mathcal{T}'\setminus\mathcal{T}$, gives
\[
\operatorname{range}(B_S^T)\subseteq\operatorname{range}({B_S'}^T)
\]
and
\[
{B_S'}^TB_S'=B_S^TB_S+E_S^TE_S\succeq B_S^TB_S.
\]
The first inclusion gives the target-span inequality. Order reversal under matrix inversion gives the target-risk inequality, while adding rows cannot decrease the smallest singular value relevant to full-state recovery.
\end{proof}

The proposition concerns additional measurements at a fixed set of locations. It does not say that fewer locations observed more often are preferable when the total number of scheduled readings is fixed; Section~\ref{sec:heldout-budgets} tests that different question on held-out data.

\subsection{Target-Aware Permanent Sampling}

TAPS builds a nested sequence of permanent sets. Let
\begin{eqnarray}
    C_S=\beta I+\sum_{v\in S}R_v^TR_v. \label{cs}
\end{eqnarray}
Beginning with $S_0=\varnothing$, TAPS selects
\begin{eqnarray}
v_j^\star
&:=&
\arg\min_{v\notin S_j}
\beta g_{\tau,c}^T
\left(C_{S_j}+R_v^TR_v\right)^{-1}
 g_{\tau,c},\nonumber\\
S_{j+1}&:=&S_j\cup\{v_j^\star\}.\label{taps}
\end{eqnarray}
Each step adds the location whose complete temporal block gives the smallest regularized target risk \eqref{rbeta}.

The criterion is a grouped, permanent-location analogue of $c$-optimal experimental design, which selects measurements for one prescribed linear combination of unknown parameters \cite{pukelsheim2006optimal}. Under the same criterion and tie-breaking, greedy $c$-optimal design and TAPS make the same selections. The distinction is the permanent space-time formulation: one decision adds the entire repeated-observation block $R_v$, while cost is charged once for the physical installation. This formulation supports the target-versus-state comparison, the target-response and finite-update marginal-gain factorization, and prospective placement when candidate sites do not yet have ground-monitor PM$_{2.5}$ histories.

We compare TAPS with target-weight ranking, weighted degree, geographic coverage, random placement, ridge leverage, D-optimal placement, and QR pivoting. D-optimal placement greedily selects $v\notin S$ to maximize
\begin{eqnarray}
\log\det\!\left(\beta I+\sum_{u\in S}R_u^TR_u+R_v^TR_v\right).\label{dopt}
\end{eqnarray}
For ridge leverage, each location receives the fixed block score
\begin{eqnarray}
\mathrm{lev}_v:=\operatorname{tr}\!\left[R_v\left(\beta I+\sum_{u\in V}R_u^TR_u\right)^{-1}R_v^T\right],\label{lev}
\end{eqnarray}

and locations are ranked by decreasing $\mathrm{lev}_v$. QR placement applies column-pivoted QR to $U_K^T$; the column pivots therefore index candidate locations, and the first $m$ pivots form the $m$-location network. Greedy candidates are scanned in the fixed candidate order, so unresolved numerical ties retain the first candidate encountered; the held-out D-optimal implementation requires a log-determinant improvement greater than $10^{-12}$ before replacing the incumbent. Target-weight ranking and weighted degree use stable decreasing sorts. Stability of D-optimal sensor placements under spectral perturbation of the retained modes is studied in \cite{YS26}.

\subsection{Target Response and Redundancy}

The central idea of TAPS appears explicitly in its marginal gain. Let
\begin{eqnarray}
    h_S&:=&C_S^{-1}g_{\tau,c},~\label{hs} \\
 q_v(S)&:=&R_vh_S,~\label{qv}\\
G_v(S)&:=&R_vC_S^{-1}R_v^T, ~\label{Gv}
\end{eqnarray}
for $C_S$ defined in \eqref{cs}, and $g_{\tau,c}$ as in \eqref{gt}, and $R_v$ as in \eqref{rv}.
\begin{proposition}[Marginal-gain factorization]
\label{prop:marginal}
The reduction in squared normalized target risk from adding $v\notin S$ is
\begin{eqnarray}
r_\beta(S;g_{\tau,c})^2
-
r_\beta(S\cup\{v\};g_{\tau,c})^2
=
A_v(S)\eta_v(S),\label{factor}
\end{eqnarray}
where
\begin{eqnarray}
A_v(S)
:=
\frac{\beta\|q_v(S)\|_2^2}
{\|g_{\tau,c}\|_2^2}\label{av}
\end{eqnarray}
is the raw target response and, when $q_v(S)\neq0$, and 
\begin{eqnarray}
    \eta_v(S)
:=
\frac{
q_v(S)^T\left(I+G_v(S)\right)^{-1}q_v(S)
}{\|q_v(S)\|_2^2}\label{eta}
\end{eqnarray}
is the finite-update correction factor. If $q_v(S)=0$, set $\eta_v(S)=0$. In all cases,
\[
0\leq\eta_v(S)\leq1.
\]
\end{proposition}

\begin{proof}
The Woodbury identity gives
\[
\begin{aligned}
\left(C_S+R_v^TR_v\right)^{-1}
&=C_S^{-1}\\
&\quad-C_S^{-1}R_v^T
\left(I+R_vC_S^{-1}R_v^T\right)^{-1}\\
&\qquad\cdot R_vC_S^{-1}.
\end{aligned}
\]
Substitution into the target-risk expression yields
\[
\frac{\beta}{\|g_{\tau,c}\|_2^2}
q_v(S)^T\left(I+G_v(S)\right)^{-1}q_v(S)
=A_v(S)\eta_v(S).
\]
Since $G_v(S)$ is positive semidefinite, the eigenvalues of $(I+G_v(S))^{-1}$ lie in $(0,1]$, proving $0\leq\eta_v(S)\leq1$.
\end{proof}

The factor $A_v(S)$ of \eqref{av} measures the candidate block's response to the current target direction $C_S^{-1}g_{\tau,c}$. The factor $\eta_v(S)$, as in \eqref{eta},  converts this raw response into the exact finite-update gain. Both factors depend on the installed network, and their product gives the reduction in squared regularized target risk. The correction factor alone is not a monotone measure of redundancy.

To isolate the finite-update correction, the raw-response ablation drops the second factor in \eqref{factor} and selects
\begin{eqnarray}
    v_{\mathrm{raw}}(S)
:=
\arg\max_{v\notin S}A_v(S).\label{vraw}
\end{eqnarray}
Raw response uses the same graph, target, diffusion model, schedule, regularization, and current set as TAPS, but omits the finite-update correction $\eta_v(S)$, as in \eqref{eta}. It therefore still conditions on the installed set through $C_S^{-1}$; the ablation removes only the correction from raw response to exact marginal gain.
Placement rules are fixed before their corresponding evaluation outcomes are used.
%
%\end{multicols}
%
%\begin{center}
%\scriptsize
%\setlength{\tabcolsep}{3pt}
%\renewcommand{\arraystretch}{1.05}
%
%\begin{tabular}{
%|>{\centering\arraybackslash}p{0.13\textwidth}
%|p{0.31\textwidth}
%|>{\centering\arraybackslash}p{0.13\textwidth}
%|p{0.31\textwidth}|}
%\hline
%\multicolumn{2}{|c|}{\textbf{Graph and Measurement Model}} &
%\multicolumn{2}{c|}{\textbf{Target-Aware Design}} \\
%\hline
%\textbf{Symbol} & \textbf{Description} &
%\textbf{Symbol} & \textbf{Description} \\
%\hline
%$G=(V,E,\omega)$ & Weighted graph
%& $c$ & Target weights \\
%
%$L=D-W$ & Weighted graph Laplacian
%& $\tau$ & Prediction horizon \\
%
%$U_K,\Lambda_K$ & Retained eigensystem
%& $g_{\tau,c}$ & Retained target direction \\
%
%$K$ & Retained dimension
%& $\beta$ & Regularization parameter \\
%
%$\mathcal{T}$ & Observation times
%& $C_S$ & Regularized information matrix \\
%
%$r=|\mathcal{T}|$ & Readings per location
%& $e_{\mathrm{tar}}(S)$ & Target-span error \\
%
%$R_v$ & Measurement block at $v$
%& $r_\beta(S;g_{\tau,c})$ & Regularized target risk \\
%
%$B_S$ & Measurement matrix for $S$
%& $A_v(S)$ & Raw target response \\
%
%$S$ & Installed locations
%& $\eta_v(S)$ & Finite-update correction \\
%\hline
%\end{tabular}
%
%\captionof{table}{Notation used throughout the TAPS framework.}
%\label{tab:notation}
%\end{center}
%
%\begin{multicols}{2}

\section{California PM$_{2.5}$ Study}
\label{sec:california}

We first evaluate TAPS on California wildfire-season PM$_{2.5}$. The study constructs a population-weighted Bay Area target and an empirical station graph, then evaluates model-based target information separately from held-out prediction of future observed concentrations.

\subsection{Study Region and Target}

We use daily federal reference or equivalent method PM$_{2.5}$ observations from EPA AirData during June to October 2018 to 2024 \cite{epaAirData}. Requiring observations in at least five of the seven wildfire seasons and a median of at least 40 observed days per season leaves 104 California candidates.

We screen four air districts for both monitoring coverage and recurrent smoke exposure. NOAA Hazard Mapping System polygons assign each station-day its strongest overlapping category (light, medium, or heavy), and official CARB boundaries define the districts \cite{noaaHMS,carbAirDistricts}. Table~\ref{tab:regions} summarizes the four-region comparison.

\begin{center}
\small
\setlength{\tabcolsep}{3pt}
\renewcommand{\arraystretch}{1.10}
\begin{tabularx}{\columnwidth}{|>{\raggedright\arraybackslash}X|r|r|r|}
\hline
\textbf{Region} & \textbf{Stations} & \shortstack{\textbf{Smoke}\\\textbf{Days}} & \shortstack{\textbf{Medium/}\\\textbf{Heavy}} \\
\hline
Sacramento Metropolitan & 5 & 45.0 & 10.0 \\
\hline
Bay Area & \textbf{16} & \textbf{41.0} & \textbf{8.0} \\
\hline
San Joaquin Valley & 17 & 34.5 & 6.0 \\
\hline
South Coast & 18 & 18.0 & 3.0 \\
\hline
\end{tabularx}
\captionof{table}{Selection of the California Target Region. Entries give qualified PM$_{2.5}$ stations and median smoke-covered days per station-year during June to October 2018 to 2024. Data are from EPA AirData, NOAA HMS, and CARB \cite{epaAirData,noaaHMS,carbAirDistricts}.\repro{01-regions.ipynb}{8}}
\label{tab:regions}
\end{center}

The Bay Area provides the best balance for the primary target: 16 qualified stations and a median of 41 smoke-covered days per station-year, including eight medium- or heavy-smoke days. Sacramento is more smoke-affected but has only five qualified stations, while South Coast has substantially less smoke exposure.

We define $c$ as population-weighted Bay Area PM$_{2.5}$. Population from the 2024 ACS B01003 table is intersected with 2024 TIGER/Line tracts and the official district boundary, with boundary tracts weighted by their fractional area inside the district \cite{censusACS2024,censusTIGER2024,carbAirDistricts}. The target represents approximately 7.41 million people in 1,758 tracts. Assigning tract population to the nearest qualified station and normalizing gives positive target weight to 21 candidates, including 16 inside the district.

\subsection{Empirical Graph and Model}

We represent empirical similarity between monitoring locations by a weighted graph. For stations $u$ and $v$, let $d(u,v)$ be geographic distance and $\rho_{uv}$ their PM$_{2.5}$ correlation. Correlations are estimated from 2018 to 2022, graph parameters are selected on 2023, and no 2024 concentration enters graph construction. For each station $u$, let $\sigma_u$ be the distance to its $k$th nearest geographic neighbor. We symmetrize the directed $k$-nearest-neighbor graph by edge union and set
\[
\omega_{uv}
=
\mathbf{1}_{\{\{u,v\}\in E\}}
\exp\left(-\frac{d(u,v)^2}{\sigma_u\sigma_v}\right)
\max\{\rho_{uv},0\}^{q}.
\]

We search $k\in\{4,5,6,8,10\}$ and $q\in\{0,1,2\}$, retain graphs that are connected on their positive-weight edges, and rank them by 2023 smoke-day normalized root-squared error, overall normalized root-squared error, and then edge count. Here the normalized error is $\sqrt{\sum(\hat y-y)^2/\sum y^2}$, distinct from the held-out nRMSE in Table~\ref{tab:heldout-models}. The selected graph uses $k=10$ and $q=2$, giving 104 vertices, 664 weighted edges, and one connected component. It is a similarity model, not a literal model of atmospheric transport. We retain $K=40$ modes, which capture 94.7\% of total training energy and 89.0\% of centered spatial energy, and use
\[
t_j\lambda_{K-1}\in\{0,0.25,0.5\},
\qquad
\tau\lambda_{K-1}=1.
\]
Here $B_{\mathrm{all}}$ denotes the measurement matrix $B_S$ for the full candidate set $S=V$ under the same observation schedule, and
\begin{eqnarray}
    \beta
=
0.01\frac{\operatorname{Tr}(B_{\mathrm{all}}^TB_{\mathrm{all}})}{K}
=
0.024502.\label{betacal}
\end{eqnarray}
Figures~\ref{fig:california-context} and~\ref{fig:california-maps} show the resulting California similarity graph, Bay Area target, and ten-location TAPS placement.

\begin{center}
\includegraphics[width=\columnwidth]{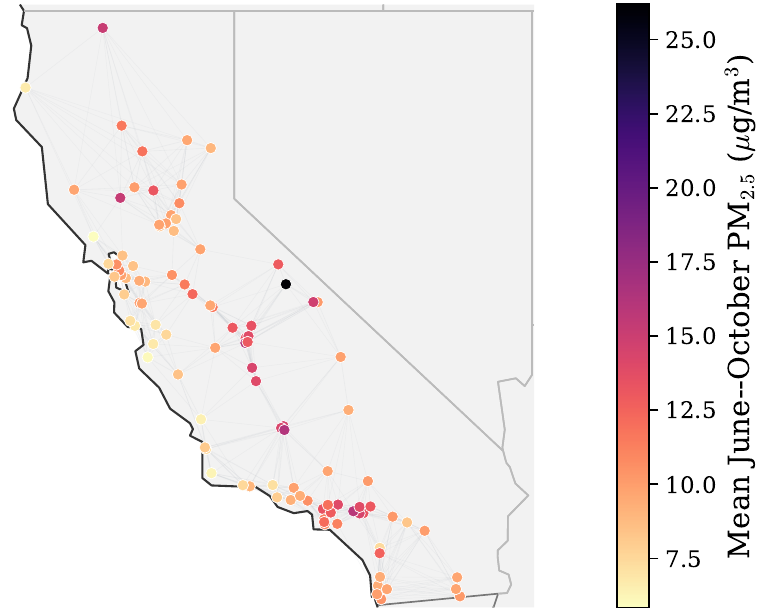}
\captionof{figure}{California wildfire-season monitoring network. Candidate stations are colored by mean June to October PM$_{2.5}$, and edges show the selected empirical similarity graph.\repro{02-graphs.ipynb}{12}}
\label{fig:california-context}
\end{center}

\begin{center}
\includegraphics[width=\columnwidth]{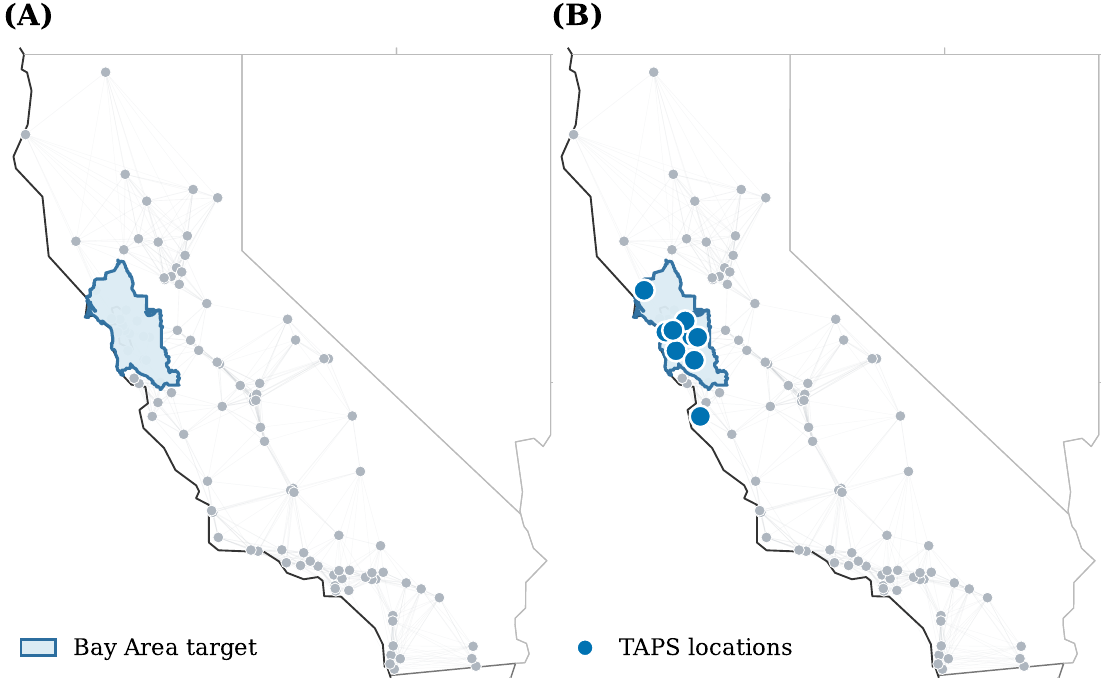}
\captionof{figure}{Bay Area target and TAPS placement. (A) The 104-station candidate graph with the Bay Area target highlighted. (B) The ten permanent locations selected by TAPS on the same graph. Locations outside the target region may be selected when they add information about the prescribed regional exposure.\repro{02-graphs.ipynb}{20}}
\label{fig:california-maps}
\end{center}

\subsection{Evaluation Design}

We use two complementary endpoints. The model-based experiments ask how much information a placement contains about the prescribed target under graph diffusion. The held-out experiments ask how well the selected measurements predict future observed PM$_{2.5}$.

\paragraph{Model-Based Analysis.}
We compare placements at the same permanent-location budget and observation schedule using target-span error, target mismatch, amplification, regularized target risk $r_\beta(S;g_{\tau,c})$, $\operatorname{rank}(B_S)$, and $\sigma_{\min}(B_S)$. Empirical risk values below report $r_\beta$; marginal gains are differences in $r_\beta^2$. The ten-location comparison includes 200 random sets generated with seed 7. Bandwidth-horizon sensitivity uses 50 random sets for each of the 12 combinations $K\in\{20,30,40,50\}$ and $\tau\lambda_{K-1}\in\{0.5,1,2\}$. We use the same settings to test the marginal-gain identity, conditional redundancy, connected regionalization, quota transfer, and geographic spacing.

\paragraph{Held-Out Prediction.}
The Bay Area region, 104-station study universe, and population-weighted target are constructed before the held-out forecasting split using the full study design described above. Within that universe, candidate eligibility and graph correlations for forecasting use 2018 to 2021. Requiring a station in at least three training years with a median of at least 40 June to October observations leaves 99 candidates, including all 21 target-reference stations. The resulting connected graph has 634 edges, with $k=10$, $q=2$, and at least 60 overlapping observations per correlation estimate. Unlike the 104-station model-based graph, this held-out graph uses one global distance scale, the median length of the symmetrized geographic edges, $\sigma_d=73.29$ km, with weight \begin{eqnarray}
   \omega_{uv}= \exp[-(d(u,v)/\sigma_d)^2]\max\{\rho_{uv},0\}^2.\label{omegacal}
\end{eqnarray} Predictor parameters are selected on 2022, models are refit on 2018 to 2022, and 2024 is reserved for the primary forecasting test.

Measurements on days $d-4$, $d-3$, and $d-2$ predict population-weighted Bay Area PM$_{2.5}$ on day $d$, without crossing calendar years. On each day, the observed target renormalizes the 21 fixed reference weights over stations with available readings. We retain days with at least 90\% of the total target weight observed, giving 135 untouched dates in 2024. For the 2023 to 2024 robustness analysis, we retain the 17 target stations with at least 75\% availability in every wildfire season from 2018 to 2024. This fixed target retains 84.37\% of the original weight and correlates $0.9996$ with the primary target over 2018 to 2022.

Validation selects $K=30$, $\tau=2$, and $\beta=10^{-4}$. For graph diffusion, measurements on $d-4$, $d-3$, and $d-2$ map to normalized times $0$, $\tau/4$, and $\tau/2$, with day $d$ evaluated at $\tau$; thus the selected $\tau=2$ is a normalized diffusion horizon. Missing graph-diffusion measurements are omitted from the regularized solve, and predictions are affine-calibrated on development data. The persistence baseline averages the available selected-network readings on $d-2$ and applies a development-fitted affine calibration. Ridge regression \cite{hoerl1970ridge} and random forest \cite{breiman2001random} use the three historical monitor dates with development-fitted median imputation, missingness indicators, and sine-cosine day-of-year features; ridge additionally standardizes its inputs. Training mean and seasonal climatology use no current monitor readings, with climatology defined by 14-day calendar bins.

We report MAE, RMSE, normalized RMSE, bias, and Pearson correlation. Paired differences and their intervals are computed at full precision and then rounded, so they need not equal the difference of the rounded values reported in the tables. Normalized RMSE divides RMSE by the sample standard deviation of the evaluated target, and bias is prediction minus truth, so positive bias denotes overprediction. MAE, RMSE, and bias are reported in $\mu\mathrm{g}/\mathrm{m}^3$; normalized RMSE and correlation are dimensionless.

We estimate forecast-error uncertainty by resampling seven-day calendar bins within each year, with gaps left in place rather than compressing retained dates into seven-observation blocks \cite{kunsch1989bootstrap}. Unless stated otherwise, intervals use 1,000 paired bootstrap replicates and nominal 95\% percentiles. The fixed 30-reading comparison uses 5,000 replicates with a Bonferroni adjustment for its two contrasts. Placement comparisons include six deterministic methods and 200 distinct random ten-location networks generated with seed 7.

\section{Results}
\label{sec:results}
In what follows we report the California results in Sections~\ref{sec:target-recovery}--\ref{sec:heldout-budgets},
separating what a placement records about the prescribed target under graph diffusion from
how well it predicts future observed concentrations, and then test both questions on
independent networks in Sections~\ref{sec:prospective} and~\ref{sec:transfer}. Throughout,
model-based and held-out quantities are kept distinct, since the two do not order the
placements in the same way.

\subsection{Target Recovery Before State Identification}
\label{sec:target-recovery}

We first show that the Bay Area target requires fewer spectral directions than the complete retained state. Figure~\ref{fig:target-dimension} shows this concentration: twelve modes contain 95\% of target energy at horizon zero, nine at normalized horizon one, and six at normalized horizon four; the retained state remains 40-dimensional.

\begin{center}
\includegraphics[width=0.92\columnwidth]{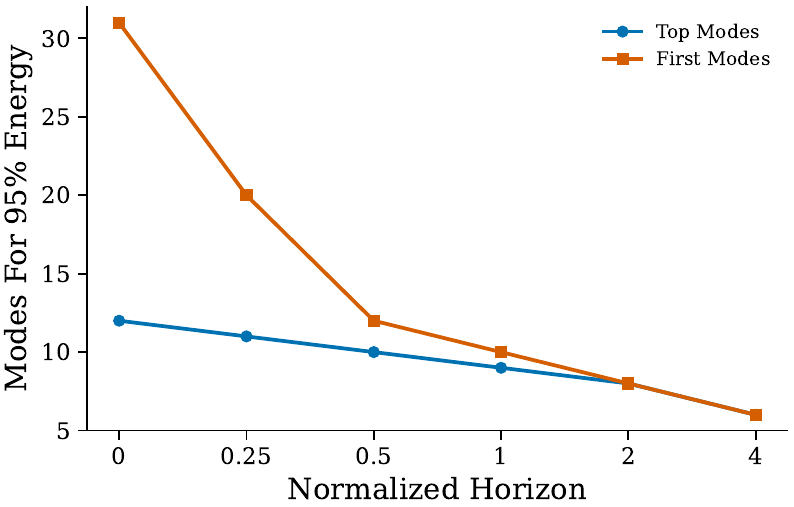}
\captionof{figure}{Spectral concentration of the Bay Area target. Top Modes gives the smallest number of target coefficients containing 95\% of target energy, while First Modes gives the number of leading low-frequency modes required for the same threshold. The complete retained state has $K=40$ modes.\repro{02-graphs.ipynb}{11}}
\label{fig:target-dimension}
\end{center}

At ten permanent locations, TAPS attains target-span error (as introduced in \eqref{etar})  below $10^{-6}$. Using a singular-value cutoff of $10^{-9}$, the measurement matrix has
\[
\operatorname{rank}(B_S)=28<40.
\]
The Bay Area target is therefore numerically recoverable even though the measurements do not identify the complete retained state.

\begin{center}
\includegraphics[width=0.9\columnwidth]{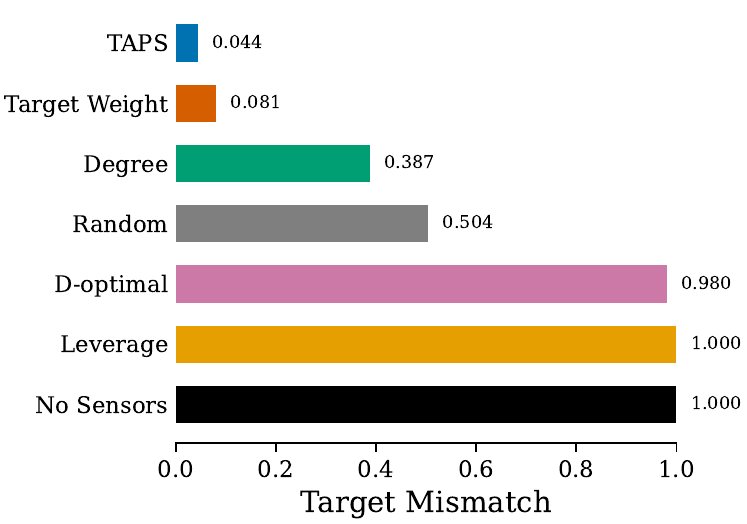}
\captionof{figure}{Target mismatch at ten permanent locations. TAPS attains mismatch $0.044$, compared with $0.081$ for target-weight ranking, $0.387$ for weighted degree, and $0.504$ averaged over 200 random networks. D-optimal and ridge-leverage placement remain near the no-sensor mismatch. Every method uses the same graph and three-observation schedule.\repro{02-graphs.ipynb}{19}}
\label{fig:ten-sensor-results}
\end{center}

Exact recovery does not by itself give a stable estimator. Figure~\ref{fig:ten-sensor-results} compares ten-location target mismatch across placements. TAPS has target mismatch $0.0440$ and amplification penalty $0.1254$. In regularized target risk the ordering is the same: target-weight placement has $r_\beta = 0.1521$ (as in \eqref{rbeta}), weighted degree $0.4388$, and the mean random network $0.5726$, against $0.1329$ for TAPS.  None of the 200 random networks matches the TAPS target mismatch.

\begin{center}
\begin{minipage}{\linewidth}
\centering
\includegraphics[width=0.96\linewidth]{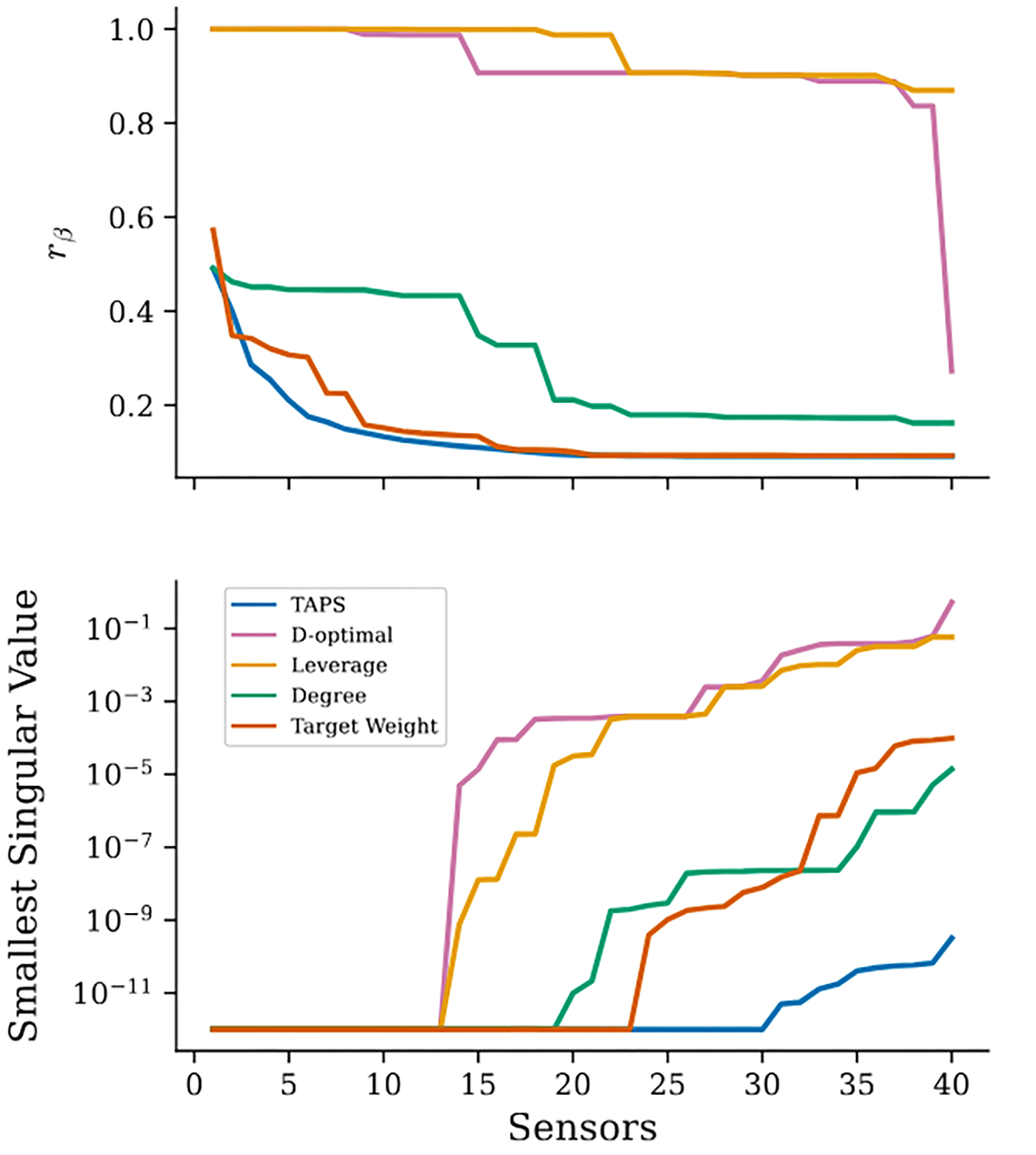}
\captionof{figure}{Target recovery and full-state conditioning across permanent-location budgets. The upper panel gives $r_\beta$, while the lower panel gives the smallest singular value of the retained-state measurement matrix. TAPS reduces $r_\beta$ at small budgets, whereas full rank does not by itself provide stable state identification.\repro{02-graphs.ipynb}{17}}
\label{fig:risk-conditioning}
\end{minipage}
\end{center}

Full rank is also distinct from stable state identification. Figure~\ref{fig:risk-conditioning} shows the resulting separation across permanent-location budgets.
D-optimal placement first reaches rank 40 at 14 locations, the counting lower bound for three readings per location,
\[
\left\lceil\frac{40}{3}\right\rceil=14.
\]
At that budget, however, $\sigma_{\min}(B_S)$ is only about $5\times10^{-6}$ and $r_\beta$ remains $0.9874$. At 40 D-optimal locations, the smallest singular value rises to $0.5231$ and $r_\beta$ falls to $0.2729$. Target recovery, algebraic full rank, and stable state identification therefore occur at different budgets.

We find the same separation across the tested bandwidths and horizons. For $K\in\{20,30,40,50\}$ and $\tau\lambda_{K-1}\in\{0.5,1,2\}$, ten-location TAPS mismatch ranges from $0.0233$ to $0.0589$ and improves on mean random placement by $81.8\%$ to $94.0\%$. The greedy target becomes exact below the full-rank counting bound in 11 of the 12 settings; Table~\ref{tab:sensitivity-summary} summarizes these ranges.

\begin{center}
\small
\renewcommand{\arraystretch}{1.10}
\begin{tabularx}{\columnwidth}{|>{\raggedright\arraybackslash}X|>{\raggedleft\arraybackslash}p{0.34\columnwidth}|}
\hline
\textbf{Quantity} & \textbf{Observed Range} \\
\hline
Ten-Location Target Mismatch & $0.0233$ to $0.0589$ \\
\hline
Ten-Location $r_\beta$ & $0.1251$ to $0.1538$ \\
\hline
Reduction Relative to Random & $81.8\%$ to $94.0\%$ \\
\hline
Greedy Exact-Target Budget & $6$ to $14$ locations \\
\hline
Below Full-Rank Lower Bound & \textbf{11 of 12 settings} \\
\hline
\end{tabularx}
\captionof{table}{Robustness of the Target-Before-State Separation. Each of the 12 settings compares TAPS with 50 random networks under the same permanent-location budget and observation schedule.\repro{02-graphs.ipynb}{21}}
\label{tab:sensitivity-summary}
\end{center}

\subsection{Repeated Observations and Permanent Sites}

We next examine how temporal reuse changes the permanent-location requirement. For $r=1,\ldots,5$, we place $r$ normalized observation times evenly over $[0,0.5]$ and recompute $\beta=0.01\,\mathrm{Tr}(B_{\mathrm{all}}^TB_{\mathrm{all}})/K$ for that schedule.

 Figure~\ref{fig:observation-count} above shows that the greedy exact-target budget falls from 33 locations with one reading per location to 15 with two, 10 with three, eight with four, and seven with five. The first added reading produces the largest ten-location reduction in mismatch, from $0.0599$ to $0.0441$; later readings change mismatch only slightly but continue to reduce the installations needed for exact recovery.
Observation timing also contributes information. With three evenly spaced readings, increasing the final normalized time from $0.1$ to $0.75$ lowers mismatch from $0.0565$ to $0.0431$ and changes the greedy exact-target budget from ten locations to nine.
This comparison concerns installation cost: the total number of measurements increases with the readings per location.

\begin{center}
\includegraphics[width=0.85\columnwidth]{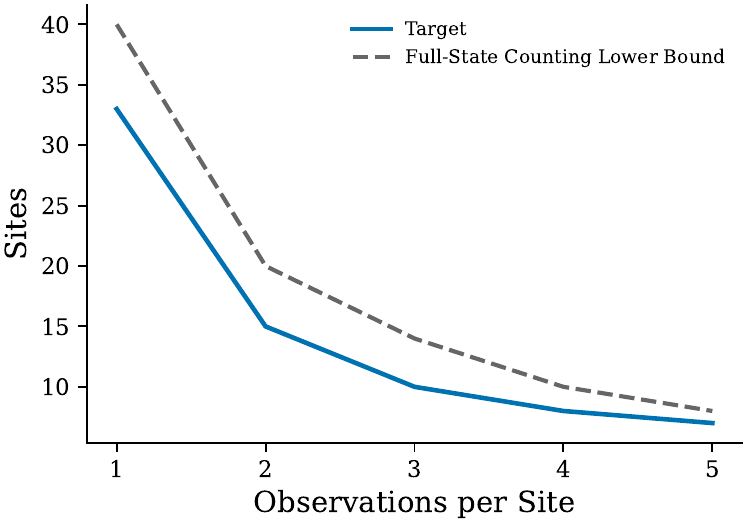}
\captionof{figure}{Repeated observations reduce the number of permanent locations required for the target. The solid Target curve gives the greedy TAPS budget for target-span error below $10^{-6}$; the dashed full-state counting lower-bound curve gives $\lceil40/r\rceil$ for $r$ readings per location. Increasing $r$ from one to five reduces the target-exact budget from 33 locations to seven.\repro{02-graphs.ipynb}{22}}
\label{fig:observation-count}
\end{center}
Because successive schedules are rescaled rather than nested and $\beta$ is recomputed for each schedule, the observed 33-to-seven trend is not a direct consequence of Proposition~\ref{prop:time-monotonicity}, which assumes nested schedules at fixed $\beta$. It also does not establish that temporal concentration is preferable under a fixed observation budget; Section~\ref{sec:heldout-budgets} tests that separate allocation problem.

\subsection{Conditional Redundancy}

We next test the mechanism behind the TAPS rule. A strong target response can add little if the installed network already represents the same information. For a candidate $v$, let
\begin{eqnarray}
    d_v(S)
:=
C_S^{-1}R_v^T
\left(I+R_vC_S^{-1}R_v^T\right)^{-1}
R_vh_S\label{dv}
\end{eqnarray}
be the Woodbury correction to the unresolved target direction for $h_S$ as in \eqref{hs}. If $C_S=F_SF_S^T$, define the correction signature
\begin{eqnarray}
   s_v(S)
:=
\begin{cases}
\dfrac{F_S^Td_v(S)}{\norm{F_S^Td_v(S)}_2},
& \norm{F_S^Td_v(S)}_2>0,\\[6pt]
0,&\text{otherwise}.
\end{cases} \label{sv}
\end{eqnarray}
The implementation uses a machine-precision floor in this normalization, so a zero correction remains the zero vector. Candidates with similar signatures reduce risk through similar target-conditioned directions.

The marginal gain of \eqref{factor} shall be denoted by
\begin{eqnarray}
    \Delta_v(S):=
r_\beta(S;g_{\tau,c})^2-r_\beta(S\cup\{v\};g_{\tau,c})^2. \label{delta}
\end{eqnarray}
 At each selection stage, let $u$ be the location chosen by TAPS. For every remaining candidate $v\neq u$, through \eqref{delta} we measure the interaction
\begin{eqnarray}
    \kappa_{uv}(S)
:=
\frac{\Delta_v(S)-\Delta_v(S\cup\{u\})}
{\Delta_u(S)+\Delta_v(S\cup\{u\})}.\label{kappa}
\end{eqnarray}
Larger $\kappa_{uv}(S)$ means that adding $u$ removes more of $v$'s subsequent marginal value. We compare $\kappa_{uv}(S)$ with the signature similarity $s_u(S)^Ts_v(S)$. Geographic, shortest-path, effective-resistance, and heat-kernel distances enter with a negative sign so that larger values denote greater similarity for every comparator.

Figure~\ref{fig:structural-results} compares correction-signature similarity with geographic, graph shortest-path, effective-resistance, and heat-kernel similarities defined by the negative corresponding distances. Across the first ten TAPS selections for four bandwidths and three horizons, giving 120 evaluated selection stages, the mean Spearman correlation between signature similarity and measured interaction is $0.889$. The correction signature is more strongly correlated with measured interaction than every tested geometry-only measure in all 120 evaluated stages.

This mechanism changes the selected locations. TAPS bypasses the strongest remaining raw-response candidate in 111 of 120 evaluated stages; whenever it does so, the chosen location has a larger finite-update correction factor. At ten locations, raw response has slightly smaller mismatch than TAPS, $0.0397$ versus $0.0440$, but a larger amplification penalty, $0.1733$ versus $0.1254$. Its $r_\beta$ is $0.1778$, a $33.8\%$ increase, and the two networks share only four locations. The finite-update correction therefore changes the physical network rather than merely reordering similar choices.
\begin{center}
\includegraphics[width=0.75\columnwidth]{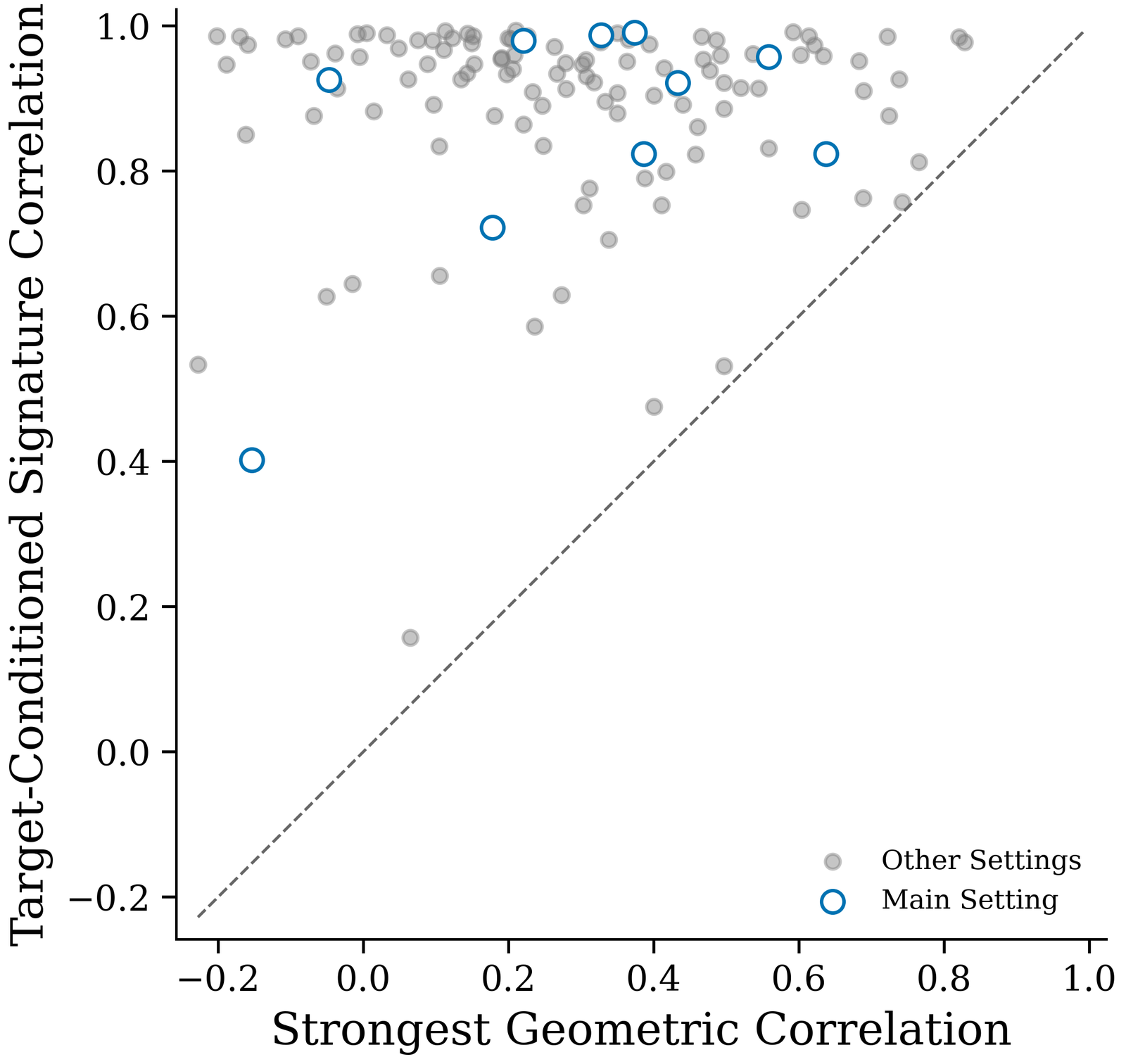}
\captionof{figure}{Target-conditioned redundancy better reflects measured interaction than geometry alone. Each point compares the Spearman correlation of measured interaction with correction-signature similarity against the strongest of four geometry-only alternatives across 120 evaluated selection stages; points above the diagonal favor the target-conditioned measure. Blue circles mark the main condition.\repro{04-bands.ipynb}{35}}
\label{fig:structural-results}
\end{center}
\paragraph{Deployment Structure.}
The selected information is geographically distributed rather than confined to the largest target weights. To summarize this structure in Figure~\ref{fig:deployment-bands}, assign each vertex the mass
\begin{equation}
\pi_v
=
\frac12\frac{A_v(\varnothing)}{\sum_{u\in V}A_u(\varnothing)}
+
\frac{1}{2n},
\label{mass}
\end{equation}
and recursively cut a minimum spanning tree with edge cost $-\log \omega_{uv}$ into three connected regions while balancing mass and station count, for $A_v(\varnothing)$ as in \eqref{av}. Each region contains at least five candidates.

\begin{center}
\includegraphics[width=0.7\columnwidth]{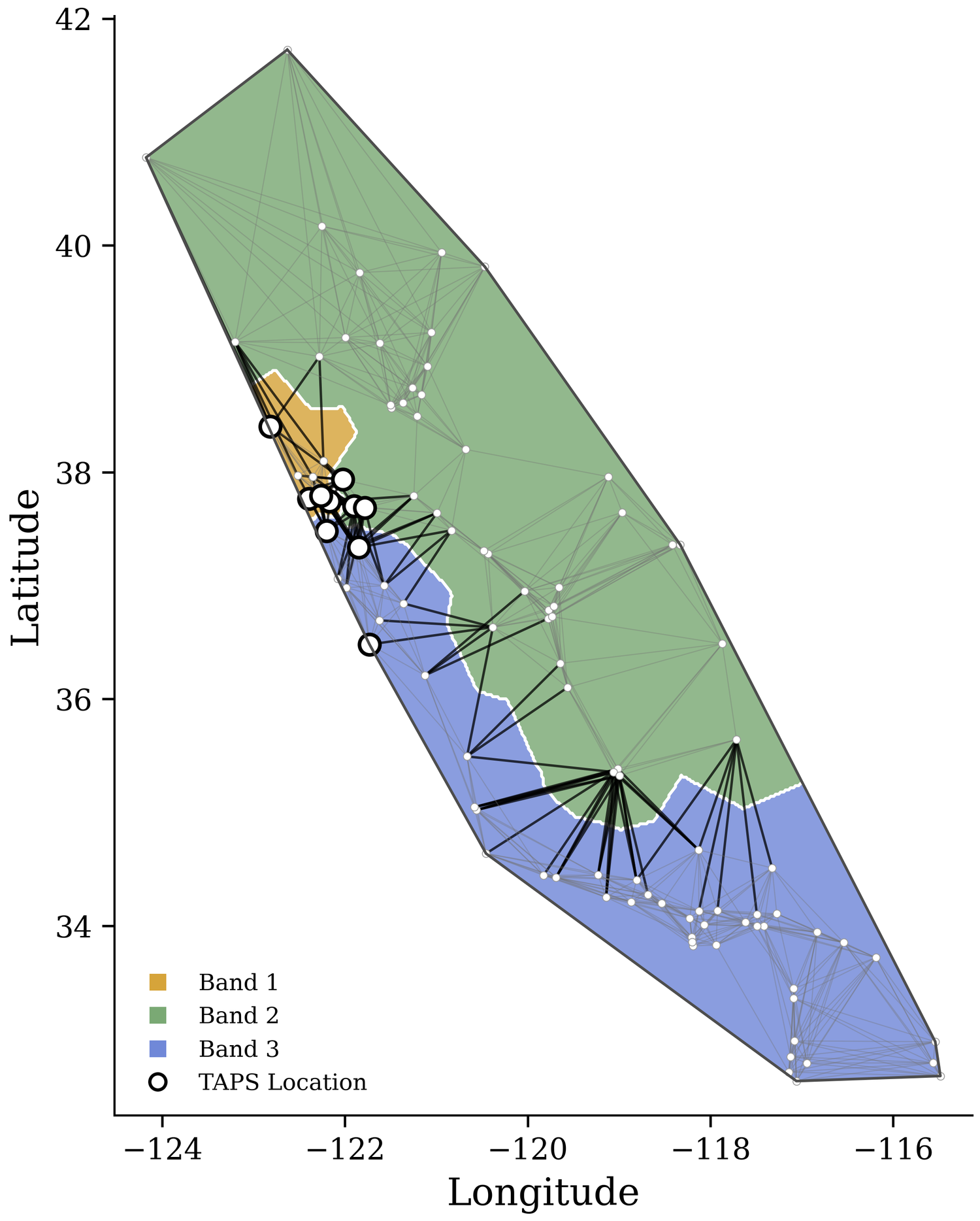}
\captionof{figure}{Connected deployment regions for the California graph. Colors show the three target-balanced bands and open circles mark TAPS locations. The colored regions are a nearest-station display of the graph partition, not interpolated PM$_{2.5}$ or physical air-shed boundaries.\repro{04-bands.ipynb}{34}}
\label{fig:deployment-bands}
\end{center}

\begin{figure*}[t]
\centering
\begin{subfigure}[t]{0.47\textwidth}
\begin{minipage}[t]{\linewidth}
\textbf{(A)}\par
\vspace{1mm}
\centering
\includegraphics[height=5.0cm,keepaspectratio]{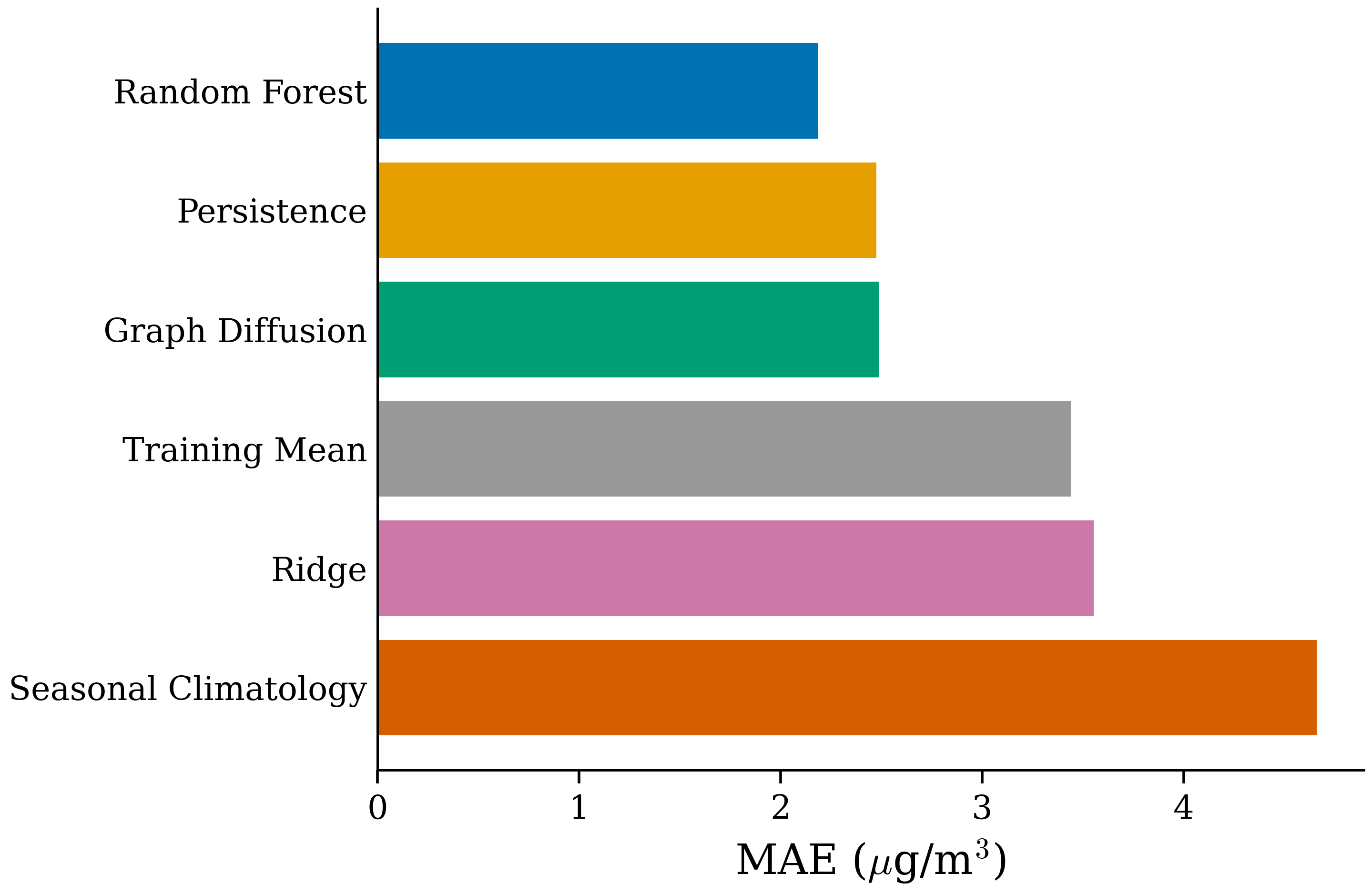}
\end{minipage}
\end{subfigure}
\hfill
\begin{subfigure}[t]{0.47\textwidth}
\begin{minipage}[t]{\linewidth}
\textbf{(B)}\par
\vspace{1mm}
\centering
\includegraphics[height=5.0cm,keepaspectratio]{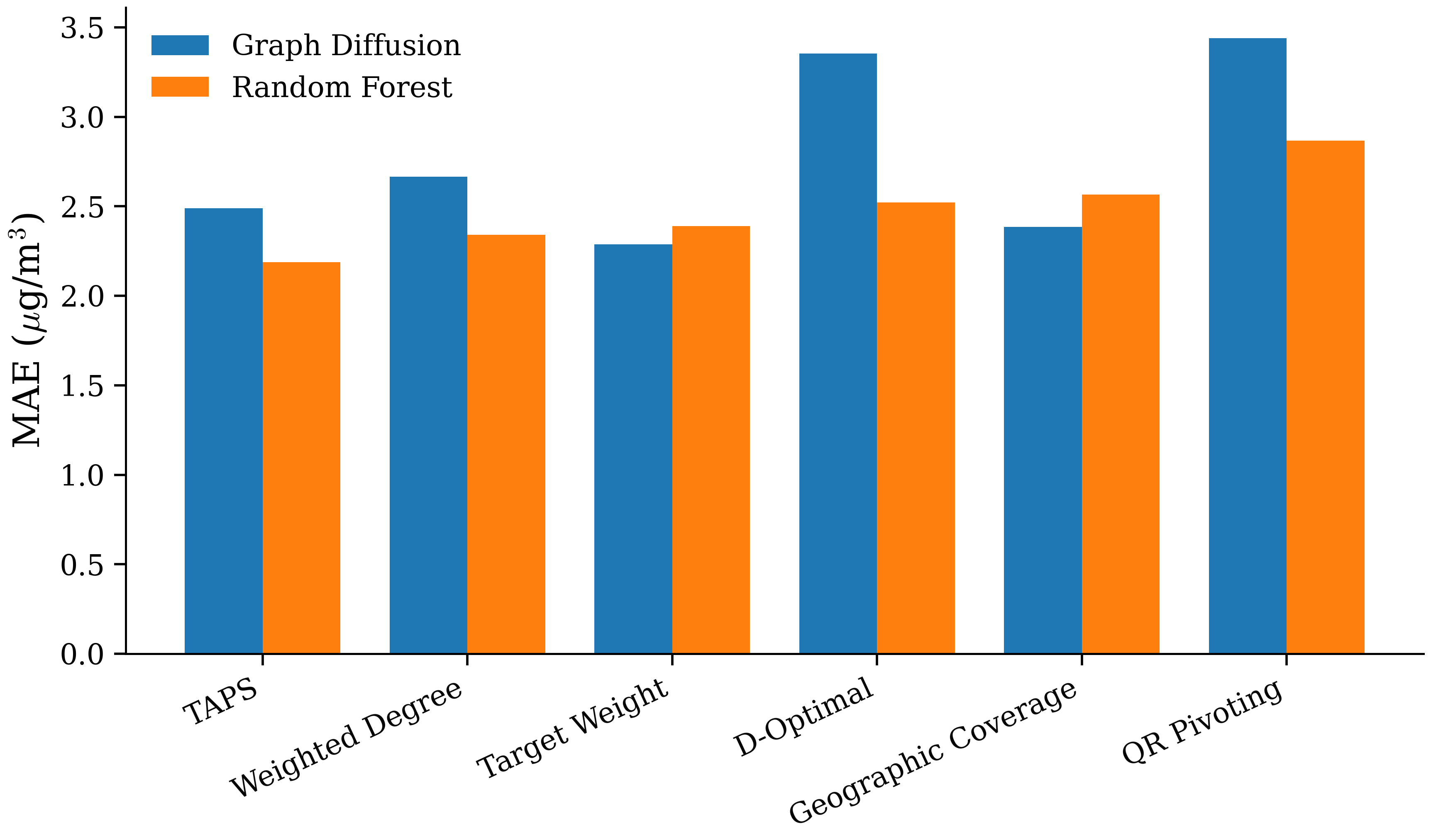}
\end{minipage}
\end{subfigure}
\caption{Held-out prediction depends on both estimator and placement. Models use 2018 to 2021 for training, 2022 for selection, 2018 to 2022 for refitting, and 135 untouched 2024 dates for evaluation. (A) MAE for six predictors on the fixed ten-location TAPS network; monitor-based predictors use observations from the selected network. (B) MAE for six ten-location placement rules under graph diffusion and random forest. TAPS and target-weight placement reverse order when only the estimator changes.\repros{05-validation.ipynb}{29 and 30}}
\label{fig:heldout-results}
\end{figure*}
The main partition contains 9, 45, and 50 stations, and TAPS allocates its ten locations as 4, 3, and 3 across the regions. Across the 12 bandwidth-horizon settings, mean node-label agreement is $0.995$ and mean adjusted Rand index is $0.990$ \cite{hubert1985partitions}. Learning a quota from the other 11 settings most often gives the same allocation of 4, 3, and 3 locations; imposing it raises $r_\beta$ by a median of $1.76\%$ and at most $7.70\%$. A 20 km spacing constraint is feasible in every setting and raises $r_\beta$ by at most $3.26\%$. Regional representation and moderate spacing can therefore be imposed at limited model cost, although spacing alone has essentially no monotone relationship with held-out error among the random networks.

\subsection{Estimator Dependence}

We next ask whether the selected measurements predict future Bay Area PM$_{2.5}$. Figure~\ref{fig:heldout-results} summarizes the estimator and placement comparisons.

On 135 untouched 2024 dates, the fixed ten-location TAPS network gives MAE $2.187\,\mu\mathrm{g}/\mathrm{m}^3$ with random forest, $2.476$ with persistence, and $2.490$ with graph diffusion. Relative to graph diffusion, random forest lowers MAE by $0.302\,\mu\mathrm{g}/\mathrm{m}^3$ (95\% CI $[0.132,0.501]$) and RMSE by $0.294\,\mu\mathrm{g}/\mathrm{m}^3$ (95\% CI $[0.111,0.532]$).

The predictor ordering is not universal. On the 2023 robustness data, persistence has the lowest MAE and the pairwise differences among persistence, graph diffusion, and random forest are not consistently resolved. We therefore identify random forest as the strongest predictor on the primary 2024 endpoint, not as a generally superior forecasting model. Table~\ref{tab:heldout-models} reports the complete error profile.

\begin{table*}[t]
\centering
\small
\setlength{\tabcolsep}{5pt}
\renewcommand{\arraystretch}{1.10}
\begin{tabular}{|l|r|r|r|r|r|}
\hline
\textbf{Model} & \textbf{MAE} & \textbf{RMSE} & \textbf{nRMSE} & \textbf{Bias} & \textbf{Correlation} \\
\hline
Random Forest & \textbf{2.187} & \textbf{2.763} & \textbf{0.899} & 0.280 & \textbf{0.443} \\
\hline
Persistence & 2.476 & 3.041 & 0.989 & 1.086 & 0.375 \\
\hline
Graph Diffusion & 2.490 & 3.057 & 0.994 & 1.065 & 0.357 \\
\hline
Training Mean & 3.439 & 4.012 & 1.305 & 2.591 & \textemdash{} \\
\hline
Ridge & 3.555 & 4.066 & 1.322 & 2.657 & 0.249 \\
\hline
Seasonal Climatology & 4.661 & 5.866 & 1.908 & 2.846 & $-0.245$ \\
\hline
\end{tabular}
\caption{Prediction from the Fixed Ten-Location TAPS Network. Graph diffusion, ridge, and random forest use all three historical monitor dates; persistence uses the latest selected-network readings, while training mean and seasonal climatology use no current monitor inputs. MAE, RMSE, and bias are reported in $\mu\mathrm{g}/\mathrm{m}^3$; nRMSE and correlation are dimensionless.\repro{05-validation.ipynb}{20}}
\label{tab:heldout-models}
\end{table*}

Changing only the placement produces the same estimator dependence. Under graph diffusion, target-weight placement has MAE $2.287$, compared with $2.490$ for TAPS; the target-weight-minus-TAPS difference is $-0.203\,\mu\mathrm{g}/\mathrm{m}^3$ (95\% CI $[-0.395,-0.0005]$), while the RMSE difference is unresolved. Under random forest, TAPS has MAE $2.187$, compared with $2.389$ for target weight. The TAPS-minus-target-weight differences are $-0.201\,\mu\mathrm{g}/\mathrm{m}^3$ for MAE (95\% CI $[-0.317,-0.075]$) and $-0.162\,\mu\mathrm{g}/\mathrm{m}^3$ for RMSE (95\% CI $[-0.275,-0.041]$). The same two networks therefore reverse order when only the estimator changes.

These deterministic placements are not typical random networks. Conditional on the fixed 2024 outcomes, TAPS under random forest outperforms 194 of 200 random ten-location networks, or $97.0\%$ (exact binomial interval $93.6\%$ to $98.9\%$). Under graph diffusion, target-weight placement outperforms 180 of 200 random networks, or $90.0\%$ (interval $85.0\%$ to $93.8\%$). These intervals describe sampling over the random-placement distribution, not variation across future years. We do not treat the best random network as a prospective competitor because it can be identified only after observing the 2024 outcomes.

\subsection{Space-Time Allocation}
\label{sec:heldout-budgets}

We find the same dependence across permanent-location budgets. Figure~\ref{fig:heldout-budget-curves} shows the complete three-through-30-location curves. Under random forest, ten-location TAPS has MAE $2.187\,\mu\mathrm{g}/\mathrm{m}^3$, and no tested non-TAPS placement from three through 30 locations has lower observed error; weighted degree at six locations is closest, with MAE $2.210$. This is a post hoc comparison across methods and budgets. When the competing method and budget are reselected within each seven-day block-bootstrap replicate, the TAPS-minus-competitor difference is $-0.022\,\mu\mathrm{g}/\mathrm{m}^3$ (95\% CI $[-0.109,0.164]$). TAPS is therefore competitive through 30 locations, but the data do not establish a minimum forecasting budget. Graph diffusion produces a different placement ordering.

\begin{figure*}[t]
\centering
\includegraphics[width=0.96\textwidth]{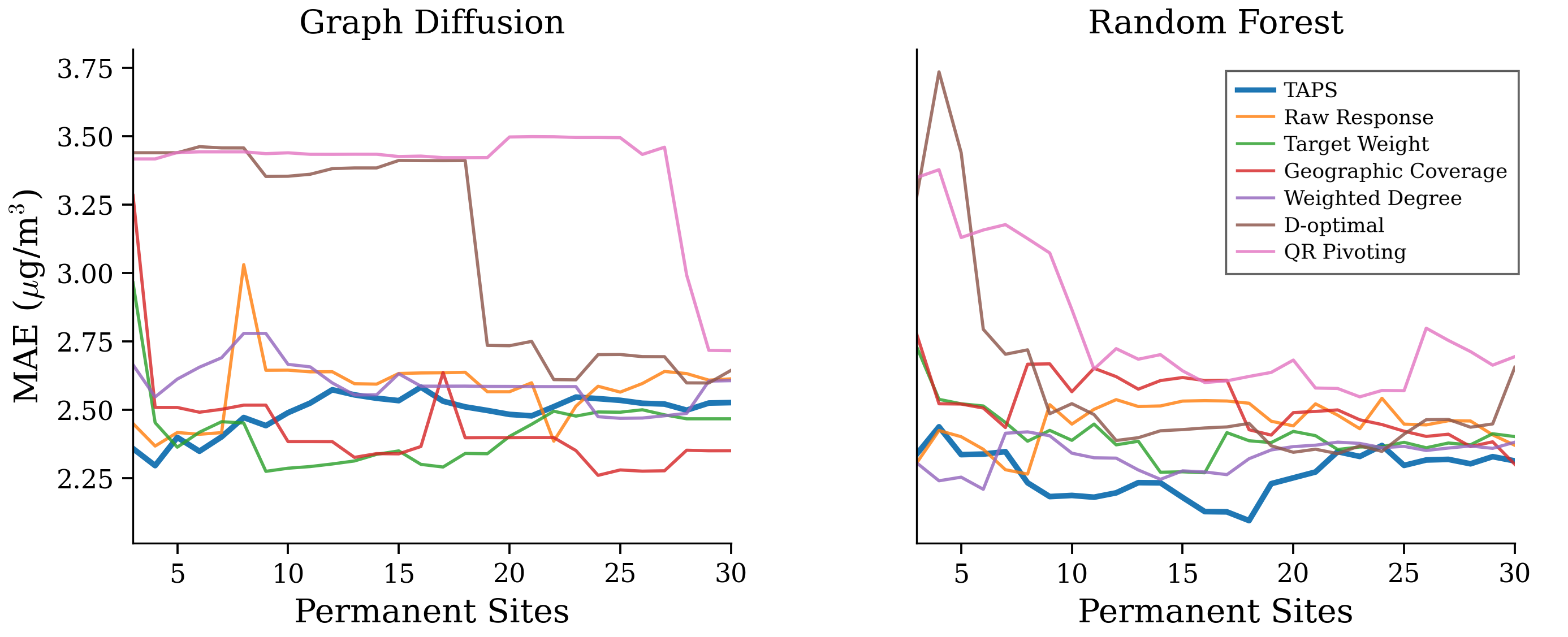}
\caption{The preferred placement changes with the estimator across sensor budgets. Each curve follows one nested placement sequence from three through 30 locations under the fixed California held-out protocol. The left panel uses graph diffusion and the right panel uses random forest. These test-set curves describe observed 2024 performance, not validation-selected sensor requirements.\repro{09-budgets.ipynb}{16}}
\label{fig:heldout-budget-curves}
\end{figure*}

We finally hold the number of scheduled readings fixed and change their allocation across space and time. Table~\ref{tab:equal-reading-budget} compares 30 locations observed once, 15 observed twice, and ten observed three times, rebuilding TAPS for each schedule. Missing monitor values mean that equal scheduled budgets need not give exactly equal realized inputs.

\begin{center}
\small
\setlength{\tabcolsep}{3pt}
\renewcommand{\arraystretch}{1.10}
\begin{tabular}{|l|r|r|r|}
\hline
\textbf{Design} & \shortstack{\textbf{Graph}\\\textbf{MAE}} & \shortstack{\textbf{RF}\\\textbf{MAE}} & $r_\beta$ \\
\hline
$30\times1$ & \textbf{2.272} & 2.349 & \textbf{0.0114} \\
\hline
$15\times2$ & 2.317 & 2.383 & 0.0121 \\
\hline
$10\times3$ & 2.490 & \textbf{2.187} & 0.0138 \\
\hline
\end{tabular}
\captionof{table}{Space-Time Allocation at a Fixed Budget of 30 Scheduled Readings. A design $m\times r$ uses $m$ permanent locations and $r$ readings per location. MAE is evaluated on the untouched 2024 California target in $\mu\mathrm{g}/\mathrm{m}^3$; $r_\beta$ is the TAPS regularized target risk for the corresponding schedule.\repros{09-budgets.ipynb}{25 and 29}}
\label{tab:equal-reading-budget}
\end{center}

For each predictor, the two contrasts of $10\times3$ against $30\times1$ and $15\times2$ use the two-comparison Bonferroni-adjusted intervals described above. Graph diffusion favors spatial breadth. The $30\times1$ design has MAE $2.272$, compared with $2.490$ for $10\times3$; the $10\times3$ minus $30\times1$ interval is $[0.044,0.409]$, and the $10\times3$ minus $15\times2$ interval is $[0.028,0.331]$.

Random forest reverses this ordering. The $10\times3$ design has MAE $2.187$, compared with $2.349$ for $30\times1$ and $2.383$ for $15\times2$. Its advantage over $15\times2$ is resolved, with interval $[-0.393,-0.003]$, whereas the comparison with $30\times1$, $[-0.338,0.029]$, remains unresolved. Thus repeated observations can reduce physical infrastructure when installations are the constrained resource, but the best allocation at a fixed measurement budget depends on the estimator.

\subsection{Blind Prospective Placement}
\label{sec:prospective}

We next consider the greenfield setting in which permanent locations must be chosen before candidate sites have produced any ground-monitor PM$_{2.5}$ history. Satellite covariates have previously informed PM$_{2.5}$ placement directly \cite{changsilva2025satellite}; here we use exogenous environmental information to construct the graph on which TAPS operates.

We generate four 100-location candidate universes by uniform, population-weighted, spatially stratified, and clustered sampling, and use the spatially stratified universe as the primary design. Environmental signatures use monthly median MAIAC $0.55\,\mu\mathrm{m}$ aerosol optical depth \cite{maiacMCD19A2}, monthly ERA5 temperature, dew point, pressure, wind components, and precipitation \cite{era5}, and USGS 3DEP elevation \cite{usgs3dep}. Each feature coordinate is standardized across candidates; the AOD and meteorological blocks are additionally scaled by the square root of their respective dimensions. Missing AOD values are left missing and pairwise environmental distances are computed from the available coordinates without imputation.

The graph uses the union of symmetric geographic 10-nearest-neighbor edges and the geographic minimum spanning tree. If $d_{uv}$ and $e_{uv}$ denote geographic and environmental distance, respectively, in lieu of \eqref{omegacal}, we have
\[
\omega_{uv}
=
\exp\!\left[-\left(\frac{d_{uv}}{\sigma_d}\right)^2\right]
\exp\!\left[-\left(\frac{e_{uv}}{\sigma_e}\right)^2\right],
\]
where $\sigma_d$ and $\sigma_e$ are the median geographic and environmental distances over the selected edges. For the primary spatially stratified graph, $\sigma_d=115.54$ km and $\sigma_e=1.5398$. ESA WorldCover supplies the land mask \cite{worldcover2021}, while GHSL population data determine candidate and target weights \cite{ghsl2023}. We fix $K=30$, normalized observation times $\{0,0.25,0.5\}$, target horizon $1$, and regularization scale $0.01$ (recall \eqref{betacal}). No historical ground-monitor PM$_{2.5}$ from the hypothetical locations enters placement.

\begin{figure*}[t]
\centering
\includegraphics[width=1.0\textwidth]{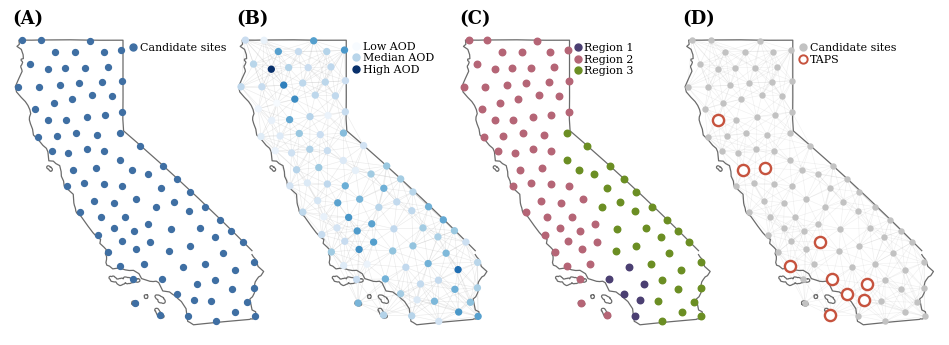}
\caption{Prospective TAPS placement without candidate-site PM$_{2.5}$ histories. (A) One hundred spatially stratified candidate locations. (B) Environmental-similarity graph constructed from multi-date MAIAC aerosol optical depth, ERA5 meteorology, USGS elevation, and geographic separation \cite{maiacMCD19A2,era5,usgs3dep}. (C) Connected target-balanced regions. (D) Ten locations selected by TAPS. Population and land constraints are derived from GHSL and ESA WorldCover \cite{ghsl2023,worldcover2021}.\repro{07-prospective.ipynb}{25}}
\label{fig:prospective-pipeline}
\end{figure*}

\begin{figure*}[t]
\centering
\includegraphics[width=1.0\textwidth]{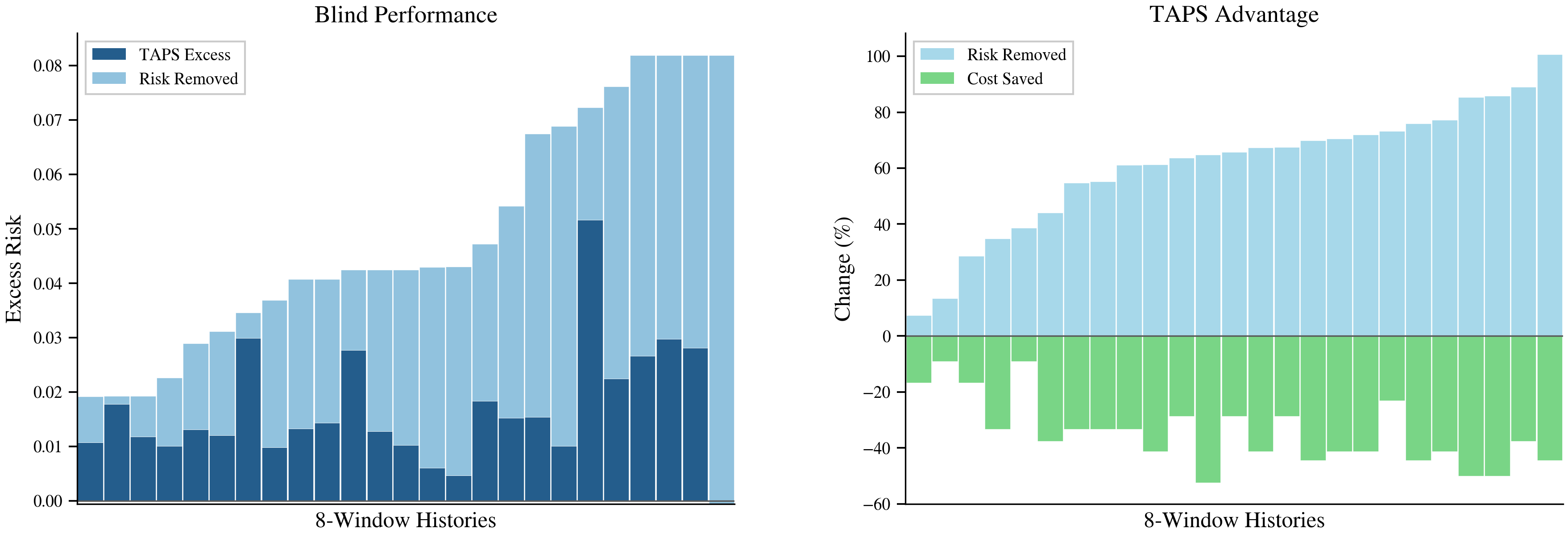}
\caption{Blind prospective performance under 25 eight-window pre-deployment histories. The practical baseline is the lower-$r_\beta$ ten-location design between target weight and raw response on the withheld 2025 model. Left: dark bars show signed blind-TAPS $r_\beta$ excess above the 2025-only oracle-information greedy benchmark, while light bars show the baseline-to-TAPS difference in $r_\beta$. Right: upper bars show the fraction of baseline $r_\beta$ excess removed by TAPS, while lower bars show the corresponding reduction in matching permanent-monitor count and annualized monitor cost under the linear EPA benchmark. TAPS outperforms the practical baseline in all 25 histories, removing a median 65.7\% of its $r_\beta$ excess while reducing the matching monitor requirement by a median 37.5\%.\repro{10-blind.ipynb}{19}}
\label{fig:blind-prospective}
\end{figure*}

We test this design across 100 blind pre-deployment histories. All 2025 windows are excluded from placement. From the 20 monthly windows (June to October, 2021 to 2024), we construct 25 histories using each of 4, 8, 12, and 16 windows. TAPS and five competing methods select networks from pre-2025 information, and the frozen networks are evaluated under the same held-out 2025 model. The held-out graph is built from 2025 MAIAC and ERA5 covariates, fixed geography, and elevation, not ground-monitor PM$_{2.5}$; the target weights remain fixed. The 2025-only benchmark is the ten-location greedy TAPS design constructed directly from that held-out model, so it is an oracle-information greedy benchmark rather than a global optimum. The experiment measures sensitivity to the available pre-deployment history, not variation across 100 independent future years.

TAPS has the lowest $r_\beta$ among all tested ten-location designs in 92 of the 100 histories. Pairwise, it outperforms target-weight placement in 95 histories, raw response in 96, and geographic coverage, D-optimal placement, and QR pivoting in all 100.

The amount of environmental history matters. With four windows, the median blind-to-benchmark $r_\beta$ ratio is $1.166$ and the 90th percentile is $2.456$. With 8 to 16 windows, the median ratios fall to $1.070$ to $1.082$; TAPS has the lowest $r_\beta$ in 73 of 75 histories, and 72 of 75 designs lie within 20\% of the 2025-only greedy benchmark. Figure~\ref{fig:blind-prospective} focuses on the 25 eight-window histories, where TAPS outperforms the lower-$r_\beta$ ten-location design between target weight and raw response in all 25. This practical baseline is a post hoc comparison envelope on the withheld model, not a placement rule chosen prospectively. For these 25 histories, the fraction of baseline excess removed is $(R_B-R_T)/(R_B-R_O)$, where $R_B$, $R_T$, and $R_O$ denote the practical-baseline, blind-TAPS, and 2025-only greedy risks, respectively; $R_B-R_O$ is positive in every case, and values above 100\% occur when blind TAPS beats the greedy benchmark. The 20\% level is descriptive rather than prespecified.

This $r_\beta$ difference also changes the required infrastructure. For each history and placement method, we search prefixes from one through 30 locations and record the first whose $r_\beta$ is no greater than the corresponding TAPS-10 $r_\beta$, up to numerical tolerance $10^{-12}$. All target-weight and raw-response cases match within this range. Their median matching budgets are 15 and 17 locations, so TAPS-10 uses 33.3\% and 41.2\% fewer permanent monitors. At the EPA benchmark of \$22,456 per continuous PM$_{2.5}$ monitor per year, the corresponding five- and seven-monitor differences have illustrative annualized equivalents of \$112,280 and \$157,192, respectively \cite{epaOIG2025monitoring}. These values translate the observed site-count differences under a linear cost assumption and are not demonstrated deployment savings.

For Figure~\ref{fig:blind-prospective}, the 37.5\% median uses one lower-$r_\beta$ comparator per eight-window history, not both comparators pooled. The 33.3\% and 41.2\% figures above instead use separate median matching budgets of 15 and 17 across all 100 histories.

\subsection{Transfer Results}
\label{sec:transfer}

We next test whether the main distinction survives on independent regulatory networks. The transfer studies ask two separate questions: whether the prescribed target remains easier to recover than the retained state, and whether low graph-model risk identifies the best held-out placement.

\paragraph{Canada.}
We construct population-weighted Vancouver and Toronto targets from National Air Pollution Surveillance observations and 2021 census data \cite{ecccNAPS,statcan2021profile,statcan2021tracts}. Candidate eligibility uses 2018 to 2022 June to October data: a station must appear in at least four seasons and have a median of at least 40 daily observations per season, leaving 55 British Columbia and 43 Ontario candidates. Graph correlations use 2018 to 2022, graph parameters are selected on 2023, and 2024 is held out. The selected British Columbia graph uses $k=8$, $q=0.5$, and $\sigma_d=104.16$ km; Ontario uses $k=8$, $q=2$, and $\sigma_d=151.52$ km. The retained dimensions are $K=23$ and $K=34$, respectively, with normalized observation times $\{0,0.25,0.5\}$, target horizon $1$, and $\beta=0.01\,\mathrm{Tr}(B_{\mathrm{all}}^TB_{\mathrm{all}})/K$. Along the TAPS prefixes, the Vancouver target becomes numerically exact at four locations and the Toronto target at eleven. At ten locations, the measurement ranks are $19<23$ and $30<34$, respectively, so neither retained state is identifiable. The counting lower bounds are $\lceil 23/3 \rceil = 8$ and $\lceil 34/3 \rceil = 12$, so both targets meet the numerical recovery criterion below the full-state counting lower bound. TAPS also outperforms all 200 random networks under $r_\beta$ in both settings.

For held-out prediction, observations on $d-4$, $d-3$, and $d-2$ predict day $d$. The observed target renormalizes its fixed weights over available reference stations and requires at least 85\% target-weight coverage. Predictor selection uses available 2018 to 2022 dates with 2023 validation, after which models are refit through 2023 for the 2024 test. This gives 147 validation and 149 test dates for Vancouver, and 138 validation and 126 test dates for Toronto. Graph diffusion omits missing measurement rows and is affine-calibrated on development data. Random forest uses median imputation with missingness indicators and sine-cosine day-of-year features; validation selects depth 4 and feature fraction $0.5$ for both targets, with minimum leaf sizes 2 for Vancouver and 5 for Toronto, and final fits use 500 trees.

Held-out prediction gives a different ordering. In Vancouver, random forest favors TAPS with MAE $1.778\,\mu\mathrm{g}/\mathrm{m}^3$, whereas graph diffusion favors weighted degree at $1.913$, compared with $1.975$ for TAPS. In Toronto, graph diffusion favors TAPS at $2.878$, while random forest favors D-optimal placement at $2.609$. Under graph diffusion, TAPS outperforms only $36\%$ of the random Vancouver networks and $13\%$ of the random Toronto networks despite outperforming all of them under $r_\beta$. The graph criterion therefore preserves target information without determining held-out forecast error.

\paragraph{England.}
England gives the clearest external target-state separation. We use Defra's Automatic Urban and Rural Network \cite{defraAURN}, with 2018 to 2022 used for qualification and graph construction, 2023 for validation, and 2024 held out. A candidate must appear in at least four 2018 to 2022 June to October seasons, have a median of at least 40 valid days per season, and map to an English ONS region; each daily PM$_{2.5}$ value requires at least 18 valid hourly observations. These rules leave 52 candidates. The population-weighted South East target represents approximately $9.278$ million residents across 1,119 Middle Layer Super Output Areas, and 17 stations receive positive target weight \cite{onsTS001,onsMSOA}.

The retained model has $K=38$ modes and effective target dimension $d_{\mathrm{eff}}(\sqrt{0.05})=10$ (see \eqref{deff}), using the same 95\% target-energy threshold as Figure~\ref{fig:target-dimension}. With three readings per location, the ten-location TAPS matrix has rank $30<38$, yet its target-span error is below $10^{-6}$ (as introduced in \eqref{etar}). Along the TAPS sequence, full-state rank is first reached at 22 locations. The counting lower bound is $\lceil38/3\rceil=13$, so the 22-location value is an achieved budget for this greedy sequence rather than a lower bound over all possible subsets.

England also exposes the limit of the greedy prefix. At ten locations, raw response has lower $r_\beta$ than TAPS, $0.1843$ versus $0.2036$, although both outperform all 200 random networks. On 100 untouched 2024 dates, the fixed TAPS network gives MAE $2.304\,\mu\mathrm{g}/\mathrm{m}^3$ with random forest, $2.325$ with persistence, and $2.428$ with graph diffusion; the random-forest advantage over graph diffusion is unresolved.

Changing the placement again changes the ranking. Under graph diffusion, raw response attains MAE $2.174$, compared with $2.428$ for TAPS, with a paired difference of $-0.254\,\mu\mathrm{g}/\mathrm{m}^3$ (95\% CI $[-0.460,-0.095]$); raw response outperforms $89\%$ of random networks, compared with $69\%$ for TAPS. Under random forest, D-optimal placement is numerically best at $2.233$, followed by raw response at $2.248$ and TAPS at $2.304$, although the D-optimal advantage over TAPS is unresolved.

The transfer results preserve the target-state separation but reject a universal placement ranking. TAPS can identify compact target-informative networks across regions, while the best held-out network remains conditional on the downstream estimator.

\section{Discussion and Outlook}
\label{sec:discussion}
\label{sec:conclusion}

TAPS separates the information needed for one prescribed future target from that needed to identify the full retained state. This distinction is visible in both California and England, where the target is recoverable before full-state identification. Repeated observations can further reduce the number of installations: under the tested rescaled California schedules, the greedy exact-target requirement falls from 33 locations with one reading each to seven with five. The marginal-gain factorization \eqref{factor} also matters operationally. Omitting its finite-update correction changes the selected network and raises California $r_\beta$, as defined in \eqref{rbeta},  by $33.8\%$ despite slightly improving mismatch.

The prospective experiment shows that the same criterion can be used before candidate sites have ground-monitor PM$_{2.5}$ histories. TAPS has the lowest tested ten-location $r_\beta$ in 92 of 100 blind histories, while target-weight and raw-response designs require median budgets of 15 and 17 locations to match TAPS-10. These model-based gains do not imply a universal forecasting ranking. California, Canada, and England all show estimator-dependent reversals, and at a fixed 30-reading budget graph diffusion favors broader spatial coverage while random forest favors repeated observations at ten locations. Placement should therefore be evaluated with the estimator and loss used downstream.

The graphs encode empirical or environmental similarity rather than atmospheric transport; the primary California forecast test contains one untouched year; and the prospective study is conditional on a fixed candidate universe without site-feasibility constraints. The greedy sequence is not proved globally optimal, and in England the raw-response ablation attains lower ten-location $r_\beta$ than TAPS; the held-out budget curves do not establish a prospectively selected minimum network. Future work should incorporate operating constraints, richer dynamics such as coupling to stochastic fire-spread models from statistical mechanics \cite{drossel1992forest,clar1996forest}, additional regulatory networks, and prospective field deployment. Code and results are available at \tapsrepo.

\section*{Declarations}

\noindent\textbf{Acknowledgements.} 
 LPS is grateful to Magdalen College and the Mathematical Institute at the University of
Oxford for their hospitality  whilst this work was completed. 

\smallbreak
\noindent\textbf{Competing interests.} The authors have no relevant financial or non-financial interests to disclose.

\noindent\textbf{Author contributions.}
Conceptualization: Michelle Lin, Laura P. Schaposnik; Methodology and formal analysis:  Michelle Lin; Software and numerical investigation:  Michelle Lin; Writing: Michelle Lin, Laura P. Schaposnik; Supervision: Laura P. Schaposnik. Both authors read and approved the final manuscript.

\smallbreak
\noindent\textbf{Funding} The research of LPS is partially supported by NSF FRG Award DMS2152107 and NSF CAREER
Award DMS 1749013, and also supported in part by grants from the NSF (DMS2235451) and
Simons Foundation (MP-TMPS-00005320) to the NSF-Simons National Institute for Theory and
Mathematics in Biology (NITMB), and by a Simons Travel Award.
 
\small
\setlength{\bibsep}{2pt plus 0.3ex}
\bibliographystyle{plainnat}
\bibliography{references}
\end{multicols}

\end{document}